\documentclass[10pt]{article}

\usepackage[letterpaper, hmargin=1in, top=1in, bottom=1.2in, footskip=0.6in]{geometry}

\usepackage{titlesec}
\titleformat{\subsection}[runin]{\normalfont\bfseries}{\thesubsection.}{.5em}{}[.]\titlespacing{\subsection}{0pt}{2ex plus .1ex minus .2ex}{.8em}
\titleformat{\subsubsection}[runin]{\normalfont\itshape}{\thesubsubsection.}{.3em}{}[.]\titlespacing{\subsubsection}{0pt}{1ex plus .1ex minus .2ex}{.5em}
\titleformat{\paragraph}[runin]{\normalfont\itshape}{\theparagraph.}{.3em}{}[.]\titlespacing{\paragraph}{0pt}{1ex plus .1ex minus .2ex}{.5em}

\usepackage{amsmath}
\usepackage{amssymb}
\usepackage{amsfonts}
\usepackage{latexsym}
\usepackage{amsthm}
\usepackage{amsxtra}
\usepackage{amscd}
\usepackage{bbm}
\usepackage{mathrsfs}
\usepackage{bm}
\usepackage{tensor}

\usepackage{graphicx, color}

\definecolor{darkred}{rgb}{0.9,0,0.3}
\definecolor{darkblue}{rgb}{0,0.3,0.9}

\newcommand{\nc}{\normalcolor}

\definecolor{vdarkred}{rgb}{0.7,0,0.2}
\definecolor{vdarkblue}{rgb}{0,0.2,0.7}
\usepackage[pdftex, colorlinks, linkcolor=vdarkblue,citecolor=vdarkred]{hyperref}

\usepackage[nottoc,notlof,notlot]{tocbibind}
\usepackage{cite}

\numberwithin{equation}{section}
\numberwithin{figure}{section}

\theoremstyle{plain} %plain, definition, remark
\newtheorem{theorem}{Theorem}[section]
\newtheorem*{theorem*}{Theorem}
\newtheorem{lemma}[theorem]{Lemma}
\newtheorem*{lemma*}{Lemma}
\newtheorem{corollary}[theorem]{Corollary}
\newtheorem*{corollary*}{Corollary}
\newtheorem{proposition}[theorem]{Proposition}
\newtheorem*{proposition*}{Proposition}

\newtheorem*{conjecture*}{Conjecture}

\theoremstyle{definition} %plain, definition, remark
\newtheorem{definition}[theorem]{Definition}
\newtheorem*{definition*}{Definition}

\newtheorem*{example*}{Example}
\newtheorem{remark}[theorem]{Remark}
\newtheorem*{remark*}{Remark}

\newtheorem*{assumption*}{Assumption}

\renewcommand{\b}[1]{\boldsymbol{\mathrm{#1}}} %bold
\newcommand{\bb}{\mathbb} %blackboard bold
\renewcommand{\cal}{\mathcal}

\newcommand{\ul}[1]{\underline{#1} \!\,} %underline
\newcommand{\ol}[1]{\overline{#1} \!\,} %overline

\newcommand{\txt}[1]{\text{\rm{#1}}}

\newcommand{\E}{\mathbb{E}}

\newcommand{\C}{\mathbb{C}}
\newcommand{\N}{\mathbb{N}}

\newcommand{\me}{\mathrm{e}}
\newcommand{\ii}{\mathrm{i}}
\newcommand{\dd}{\mathrm{d}}
\newcommand{\col}{\mathrel{\vcenter{\baselineskip0.75ex \lineskiplimit0pt \hbox{.}\hbox{.}}}}
\newcommand*{\deq}{\mathrel{\vcenter{\baselineskip0.65ex \lineskiplimit0pt \hbox{.}\hbox{.}}}=}
\newcommand*{\eqd}{=\mathrel{\vcenter{\baselineskip0.65ex \lineskiplimit0pt \hbox{.}\hbox{.}}}}

\renewcommand{\leq}{\leqslant}
\renewcommand{\le}{\leqslant}
\renewcommand{\geq}{\geqslant}
\renewcommand{\ge}{\geqslant}
\renewcommand{\epsilon}{\varepsilon}

\newcommand{\p}[1]{({#1})}
\newcommand{\pb}[1]{\bigl({#1}\bigr)}

\newcommand{\pbb}[1]{\biggl({#1}\biggr)}

\newcommand{\abs}[1]{\lvert #1 \rvert}

\DeclareMathOperator{\tr}{Tr}

\DeclareMathOperator{\supp}{supp}
\DeclareMathOperator{\re}{Re}
\DeclareMathOperator{\im}{Im}

\newcommand{\bC}{ {\mathbb C} }
\newcommand{\bN}{ {\mathbb N} }
\newcommand{\bE}{ {\mathbb E} }

\newcommand{\bP}{ {\mathbb P} }
\newcommand{\bR}{ {\mathbb R} }

\newcommand*{\rom}[1]{\expandafter\@slowromancap\romannumeral #1@}

\title{Mesoscopic linear statistics of Wigner matrices of mixed symmetry class}
\author{Yukun He\footnote{University of Z\"{u}rich, Institute of Mathematics. Email: {\tt yukun.he@math.uzh.ch}.} }

\begin{document}
\maketitle

\begin{abstract}
We prove a central limit theorem for the mesoscopic linear statistics of $N\times N$ Wigner matrices $H$ satisfying $\mathbb E|H_{ij}|^2=1/N$ and $\mathbb E H_{ij}^2= \sigma /N$, where $\sigma \in [-1,1]$.  We show that on all mesoscopic scales $\eta$ ($1/N \ll \eta \ll 1$), the linear statistics of $H$ have a sharp transition at $1-\sigma \sim \eta$. As an application, we identify the mesoscopic linear statistics of Dyson's Brownian motion $H_t$ started from a real symmetric Wigner matrix $H_0$ at any nonnegative time $t \in [0,\infty]$. In particular, we obtain the transition from the central limit theorem for GOE to the one for GUE at time $t \sim \eta$.
\end{abstract}

\section{Introduction} \label{sec:1}
	
In this paper we consider the following class of Wigner matrices.
\begin{definition}[Wigner matrix] \label{def:Wigner}
	A \emph{Wigner matrix} is a Hermitian $N\times N$ matrix $H = H^* \in \C^{N \times N}$ whose entries $H_{ij}$ satisfy the following conditions.
	\begin{enumerate}
		\item
		The upper-triangular entries $(H_{ij} \col 1\le i\le j\le N)$ are independent.
		\item
		We have $\E H_{ij}=0$ for all $i,j$, $\E |\sqrt{N} H_{ij}|^2=1$ and $\bb E(\sqrt{N}H_{ij})^2=\sigma\in [-1,1]$ for $i \ne j$.
		\item
		There exists constants $c,C > 0$ such that $\E |\sqrt{N} H_{ij}|^{4+c-2\delta_{ij}} \le C$ for all $i,j$.
	\end{enumerate}
\end{definition}
We use the normalization $\bb E |\sqrt{N}H_{ij}|^2=1$, so that as $N \to \infty$ the spectrum of $H$ converges to the interval $[-2,2]$, and therefore its typical eigenvalue spacing is of order $N^{-1}$. We would like to study the mesoscopic linear eigenvalue statistics of $H$ in the form
\begin{equation} \label{zhat}
\hat{Z}(f)\deq \tr f\bigg(\frac{H-E}{\eta}\bigg)- N \int_{-2}^{2} \varrho(x) f\bigg(\frac{x-E}{\eta}\bigg) \mathrm{d}x\,,
\end{equation}
where $\varrho$ is the semicircle density, $f$ is a test function, $E \in (-2,2)$, and $\eta$ lies in the \emph{mesoscopic regime} $N^{-1} \ll \eta \ll 1$.

The study of mesoscopic linear statistics (MLS) was initiated in \cite{Kho1}, where the authors proved the case of Gaussian $H$  (Gaussian Orthogonal Ensemble). In \cite{Kho2}, the result was extended to Wigner matrices for $N^{-1/8} \ll \eta \ll 1$, and the scale was later pushed to $N^{-1/3}\ll \eta \ll 1$ in \cite{LS15}.  More recently in \cite{HK}, the result was pushed to all mesoscopic scales $N^{-1}\ll \eta \ll 1$. It was proved that for a Wigner matrix $H$ satisfying $\bb E (\sqrt{N}H_{ij})^2=1$, $\hat{Z}(f)$ converges in distribution to a Gaussian random variable with mean 0 and variance
\begin{equation} \label{intro_covariance}
V(f,f)\deq \frac{1}{2 \pi^2} \int \pbb{\frac{f(x) - f(y)}{x - y}}^2 \, \dd x \, \dd y
\end{equation}
as $N \to \infty$. The authors also showed in \cite{HK} that for a Wigner matrix $H$ with $\bb E (\sqrt{N}H_{ij})^2=0$,
\begin{equation}
\hat{Z}(f) \overset{d}{\longrightarrow}  \cal N\Big(0, \frac{1}{2}V(f,f)\Big)
\end{equation}
as $N \to \infty$. The results in \cite{HK} hold for general test functions $f$ subject to mild regularity and decay conditions. Very recently in \cite{LS18}, the result was extended to a class of test functions that do not vanish at infinity. In this paper, we further extend \cite{HK} by dropping the assumption $\bb E (\sqrt{N}H_{ij})^2 =0$ or $1$, and consider the general case where $\bb E (\sqrt{N}H_{ij})^2=\sigma \in [-1,1]$. As $\sigma$ goes from 0 to 1, we expect from \cite{HK} that the behaviour of $\hat{Z}(f)$ will transit from Gaussian Unitary Ensemble (GUE) statistics to Gaussian Orthogonal Ensemble (GOE) statistics. Our main result Theorem \ref{mainthm2} below shows that the transition happens sharply at $1-\sigma \sim \eta$, namely we have GUE statistics when $1-\sigma  \gg \eta$, and we have GOE statistics when $1-\sigma \ll \eta$.

One concrete example of the above model is the Dyson's Brownian motion
\begin{equation}
\dd \lambda_i=\sqrt{\frac{2}{N\beta}} \dd B_i -\lambda_i \dd t+\frac{1}{N}\sum_{j \ne i} \frac{\dd t}{\lambda_i-\lambda_j}
\end{equation}
for $i=1,2...,N$, where $\beta = 1,2$ and $B_i$ are independent standard Brownian motions. When $(\lambda_i(0))_{i=1}^N$ are the eigenvalues of a Hermitian matrix $H_0$, $(\lambda_i(t))_{i=1}^N$ have the same distribution as the eigenvalues of 
\begin{equation} \label{T1.5}
H_t=\sqrt{\me^{-t}} H_{0}+\sqrt{1-\me^{-t}}V\,,
\end{equation}
where $V$ is a Gaussian Wigner matrix of symmetry class $\beta$ (i.e.\ GOE for $\beta = 1$ and GUE for $\beta = 2$) independent of $H_0$. We are interested in the MLS of $H_t$ in the form
\begin{equation} \label{zhat'}
Y_t(f)\deq \tr f\bigg(\frac{H_t-E}{\eta}\bigg)- N \int_{-2}^{2} \varrho(x) f\bigg(\frac{x-E}{\eta}\bigg) \mathrm{d}x\,.
\end{equation}
Let $H_0$ be a real symmetric Wigner matrix and $\beta=2$, we see that $H_t$ is a Wigner matrix satisfying $\bb E |\sqrt{N}(H_t)_{ij}|^2=1$ and $\bb E (\sqrt{N}(H_t)_{ij})^2=\me^{-t}$. Thus Theorem \ref{mainthm2} shows $Y_t(f)$ transits from GOE statistics to GUE statistics when $1-\me^{-t} \sim \eta$, which is $t \sim \eta$. On the other hand, one can also choose $\beta=1$ and set $H_0$ to be a complex Hermitian Wigner matrix satisfying $\bb E  (\sqrt{N}(H_0)_{ij})^2=0$. In this case the transition from GUE statistics to GOE statistics happens at a much later time $t$, which satisfies $t+\log \eta \sim 1$. See also Remark \ref{Trmk2.3} and Corollary \ref{TCo3} below.

 The MLS for Dyson's Brownian motion was also considered \cite{JD13}, where the authors choose $\beta=2$, and let $H_0$ be either deterministic or diagonal with i.i.d.\,entries. In this case, the authors showed the transition happens at $t \sim \eta$, and they also identified the limiting behaviour for $t \in [\eta,\infty]$.   Comparing to \cite{JD13}, the model in our current paper treats Dyson's Brownian motion with random initial conditions corresponding to the eigenvalue distribution of a Wigner matrix, and with a short proof, we are able to identify the limiting behaviour for all $t \in [0,\infty]$. See also Remark \ref{Trmk2.5} below. 
 
Aside from the above articles, MLS had also been studied for invariant ensembles in \cite{BD, BEYY14,L16,BL16,LSY}, and for band matrices in \cite{EK3,EK4}.\

Comparing to \cite{HK}, the main novelty of our proof is to use the cumulant expansion formula on the imaginary part of the random variable. More precisely, instead of a standard cumulant formula that expands $\bb E [f(h,\bar{h})h]$, we expand the expression $\bb E [f(h,\bar{h})(h-\bar{h})]$ (see Lemma \ref{Tlem5.1} below). This provides the key cancellation needed to deal with the extra terms in the mixed symmetry class case. More details can be found in Section \ref{Tsec3} below.

In addition, our proof also crucially relies on the standard cumulant expansion formula (see Lemma \ref{lem:3.1} below). \nc Historically, the discovery of cumulant expansion formula goes back to the work of Barbour \cite{Barbour}, and it was first applied to random matrix theory in \cite{KKP}. Later on, cumulant expansion was also used in \cite{Kho1,Kho2,LP}. Subsequently, many problems of random matrix theory on mesoscopic scales, including the local law for Wigner matrices, were addressed using Schur's complement formula. Recently, with a more careful treatment, the cumulant method was revived in \cite{HK} as the main tool to prove the MLS of Wigner matrices, and it was later used in other works (e.g. \cite{HKR,EKS,LS1}) to prove various of results which seemed too complicated using Schur's formula. The method appears to be simple and versatile, and it has the potential to continue generating more results which were difficult to prove in the past.

To pass from Green functions to general test functions, we use an identity in complex analysis, Lemma \ref{lem5.1}, which is different from the usual Helffer-Sj\"{o}strand formula in \cite{Davies}. This avoids the usual use of cut-off functions and gives a cleaner and shorter proof. To further simplify the proof, we impose a stronger condition on the regularity of test functions, but the theorem holds for more general functions; see also Remark \ref{rmk5.2}.

The paper is organized as follows. In Section \ref{sec:2} we give the precise statements of our results, and in Section \ref{Tsec3} sketch the outline of the proof. In Section \ref{sec2.5} we gather the necessary tools for our proof. In Section \ref{sec4} we specialize Theorem \ref{mainthm2} to simplify the proof, assuming we have Lemma \ref{Tlem4.5}. Finally in Section \ref{Tsec5} we put our algebraic twist of the cumulant formula, Lemma \ref{Tlem5.1} into use, and give the proof of Lemma \ref{Tlem4.5}.

\subsection*{Conventions}
We regard $N$ as our fundamental large parameter. Any quantities that are not explicitly constant or fixed may depend on $N$; we almost always omit the argument $N$ from our notation.

\subsection*{Acknowledgements}

The author is grateful to Antti Knowles for suggesting this topic and giving various helpful comments. The author is partially supported by the Swiss National Science Foundation and the European Research Council.

\section{Results} \label{sec:2}
For fixed $r, s>0$, denote by $C^{1,r,s}(\bR)$ the space of all real-valued $C^1$-functions $f$ such that $f'$ is $r$-H\"{o}lder continuous uniformly in $x$, and $|f(x)| + |f^{\prime}(x)|=O(( 1 + |x|)^{-1-s})$. For $\mu \in [0,\infty)$, we define the quantity
\begin{equation} \label{variance}
V_{\mu}(f,g)\deq \frac{1}{4\pi^2}\int (f(x)-f(y))(g(x)-g(y))\bigg(\frac{1}{(x-y)^2}+\frac{(x-y)^2-\mu^2}{\big((x-y)^2+\mu^2\big)^2}\bigg)\, \dd x\,\dd y
\end{equation}
for all $f,g \in C^{1,r,s}(\bb R)$, and $V_{\infty}(f,g)\deq \lim\limits_{\mu \to \infty} V_{\mu}(f,g)$. We use $(Z_{\mu}(f))_{f \in C^{1,r,s}(\bR)}$ to denote the real-valued Gaussian process with mean zero and covariance 
	\begin{equation} \label{2.13}
	\bE(Z_{\mu}(f){Z_{\mu}(g)})=V_{\mu}(f,g)
	\end{equation} 
	for all $f,g\in  C^{1,r,s}(\bR)$.
	
	For $x \in \bR$, the Wigner semicircle law $\varrho$ is defined by 
	\begin{equation} \label{2.5}
	\varrho(x)\deq \frac{1}{2\pi}\sqrt{(4-x^2)_{+}}\,.
	\end{equation}
	Our main result is the weak convergence of the
	process $\hat{Z}(f)$ defined in \eqref{zhat} for $f \in C^{1,r,s}(\bR)$. We may now state our main result.
	
	\begin{theorem}
	 \label{mainthm2}
	 	Fix $\tau >0$, and let $\alpha \in [\tau,1-\tau]$, $\eta\deq N^{-\alpha}$, and $E \in [-2+\tau,2-\tau]$. Suppose
	 	\begin{equation} \label{T2.5}
	 	\frac{\sqrt{4-E^2}(1-\sigma)}{\eta} \longrightarrow \mu_*
	 	\end{equation}
	 	as $N \to \infty$, then the process $(\hat Z(f))_{f \in C^{1,r,s}(\bR)}$
	 converges in the sense of finite-dimensional distributions to $(Z_{\mu_*}(f))_{f \in C^{1,r,s}(\bR)}$ as $N \to \infty$. That is, for any fixed $p$ and $f_1,f_2,\dots,f_p \in C^{1,r,s}(\bR)$, we have
	\begin{equation}  \label{2.14}
	(\hat{Z}(f_1),\dots,\hat{Z}(f_p)) \overset{d}{\longrightarrow} (Z_{\mu_*}(f_1),\dots,Z_{\mu_*}(f_p))
	\end{equation}
	as $N \to \infty$.
		\end{theorem}

\begin{remark}
The four-moment condition (iii) in Definition \ref{def:Wigner} is optimal. In fact, for $\bb E |\sqrt{N}H_{ij}|^4\asymp N^{\epsilon}$, $\hat{Z}(f)$ is expected to fluctuate on the scale $1+ N^{\epsilon/2}\eta$, and Theorem \ref{mainthm2} becomes invalid when $\alpha<\epsilon/2$.

Our regularity condition $f \in C^{1,r,s}(\bb R)$ is not optimal; in fact, Theorem \ref{mainthm2} is expected to be true for all $f \in H^{1/2}(\bb R)$. The reason for our assumption is the use of Lemma \ref{lem5.1} (or Helffer-Sj\"{o}strand formula), which requires the function to be at least $C^{1}$. Improvements on the regularity is possible: in \cite{SoWo}, the authors used
Littlewood-Paley type decomposition show the macroscopic CLT for Wigner matrices with test functions less regular than Lipschitz functions. We do not peruse it here.

\end{remark}
\begin{remark} \label{Trmk2.3}		
	Note that $E \in [-2+\tau,2-\tau]$ implies $\sqrt{4-E^2}\sim 1$. Thus when $1-\sigma \ll \eta$, we get $\mu_*=0$, and 
	\begin{equation*}
	V_0(f,g)=\frac{1}{2\pi^2}\int \frac{(f(x)-f(y))(g(x)-g(y))}{(x-y)^2}\,\dd x\,\dd y\,;
	\end{equation*}
	when $1-\sigma\gg\eta$, we get $\mu_*=\infty$, and $V_{\infty}(f,g)=V_0(f,g)/2$. This implies the transition happens when $1-\sigma \sim \eta=N^{-\alpha}$. By embedding this result to the Dyson's Brownian motion \eqref{T1.5} and using $\eta=N^{-\alpha} \ll 1$, we easily obtain Corollary \ref{TCo3} below.
	\end{remark}

\begin{corollary} \label{TCo3}
	Let $\eta, E$ be as in Theorem \ref{mainthm2}, $H_t$ be as in \eqref{T1.5}, and $f \in C^{1,r,s}(\bR)$.
	\begin{enumerate}
		\item  When $\beta=2$ and $H_0$ is a real Symmetric Wigner matrix. Suppose
		\begin{equation} \label{TC}
		t\eta^{-1}\sqrt{4-E^2} \longrightarrow \mu_*
		\end{equation}
		as $N \to \infty$, then
		\begin{equation*}
		Y_t(f) \overset{d}{\longrightarrow} Z_{\mu_*}(f)
		\end{equation*}
		as $N \to \infty$.
		\item When $\beta=1$ and $H_0$ is a complex Hermitian Wigner matrix satisfying $\bb E  (\sqrt{N}(H_0)_{ij})^2=0$. Suppose 
		\begin{equation} \label{T2.7}
		t+\log \eta-\log \sqrt{4-E^2} \to -\log \mu_*
		\end{equation}
		as $N \to \infty$, then
		\begin{equation*}
		Y_t(f) \overset{d}{\longrightarrow} Z_{\mu_*}(f)
		\end{equation*}
		as $N \to \infty$.
	\end{enumerate} 
\end{corollary}

\begin{remark} \label{Trmk2.5}
In Theorem 2.3 of \cite{JD13}, the authors considered $H_t$ in \eqref{T1.5} with $\beta=2$ and $H_0$ deterministic. They showed when \eqref{TC} holds, $Y_t(f) \overset{d}{\longrightarrow} \cal N(0, \tilde{V}_{\mu_*}(f,f))$ as $N \to \infty$, where
\begin{equation*}
\tilde{V}_{\mu_*}(f,f)\deq \frac{1}{4\pi^2}\int \Big(\frac{f(x)-f(y)}{x-y}\Big)^2 \Big(\frac{\mu_*^2}{(x-y)^2+\mu_*^2}\Big)\, \dd x\, \dd y\,.
\end{equation*}
Note that we can rewrite $V_{\mu_*}(f,f)$ in \eqref{variance} into
\begin{equation} \label{T2.8}
\frac{1}{4\pi^2}\int (f(x)-f(y))^2\bigg(\frac{2(x-y)^2}{\big((x-y)^2+\mu_*^2\big)^2}+\frac{\mu_*^2}{(x-y)^2\big((x-y)^2+\mu^2_*\big)}\bigg)\, \dd x\,\dd y\,,
\end{equation}
and we see that the first term on \eqref{T2.8} vanishes as $\mu_* \to \infty$.
Thus a comparison of Theorem 2.3 of \cite{JD13} and Corollary \ref{TCo3} (i) shows that when we change $H_0$ from deterministic to a real Wigner matrix, we will get an additional term that corresponds to the eigenvalue fluctuation of the initial matrix. With the help of \eqref{T2.8}, we are able to identify the impacts of both the initial and the equilibrium conditions.

Corollary \ref{TCo3} (ii) also reveals some interesting phenomenon. By \eqref{T2.7} we see that the transition is attained at time $t \approx -\log \eta=\alpha \log N$. Thus comparing to the case $\beta=2$, it takes much longer time for the MLS to reach equilibrium in the case $\beta=1$. 
\end{remark}
	
One can check that all the above results can also be proved for the Green function $f(x)=(x-\ii)^{-1}$ using the method presented in this paper. For conciseness, we do not give the details here.

\section{Preliminaries} \label{sec2.5}

In this section we collect notations and tools that are used throughout the paper.

Let $M$ be an $N \times N$ matrix. We use $M^{\intercal}$ to denote the transpose of $M$, and let $M^{\intercal n} \deq (M^{\intercal})^n$, $M^{*n}\deq (M^{*})^n$, $M^{*}_{ij}\deq (M^{*})_{ij} = \ol M_{ji}$, and $M^n_{ij}\deq (M_{ij})^n$. We denote the normalized trace of $M$ by $\ul M \deq \frac{1}{N} \tr M$. For $\Sigma>0$, let $\cal N(0,\Sigma^2)$ denote the Gaussian random variable with mean 0 and variance $\Sigma^2$. We use $C_c^{\infty}(\bb R)$ to denote the set of real-valued smooth functions with compact support.

We define the resolvent of $H$ by $G(z)\deq (H-z)^{-1}$, where $\im z \ne 0$. The Stieltjes transform of the empirical spectral measure of $H$ is
\begin{equation} \label{2.4}
\underline{G(z)}\deq \frac{1}{N}\tr G(z)\,.
\end{equation}
For $z \in \bC$ with $\im z \ne 0$, the Stieltjes transform  of the Wigner semicircle law is defined by 
\begin{equation*} 
 m(z)\deq \int \frac{\varrho(x) }{x-z}\,\dd x\,,
\end{equation*}
and $m$ is the unique solution of
\begin{equation} \label{T3.3}
m(z)+\frac{1}{m(z)}+z=0
\end{equation}
satisfying $\im m(z) \im z > 0$.

We abbreviate $\langle X \rangle  \deq X-\bE X$ for any random variable $X$ with finite expectation. We first state an elementary result that will be useful for us.

\begin{lemma} \label{Tlemc}
	Let $X,Y,Z$ be random variables with finite third moments. We have
	\begin{equation*}
	\bb E XY\langle Z\rangle=\bb E \langle XY \rangle Z= \bb E X \bb E \langle Y\rangle \langle Z\rangle+ \bb E Y \bb E \langle X\rangle \langle Z\rangle+\bb E \langle X \rangle \langle Y \rangle Z-\bb E \langle X\rangle \langle Y\rangle \bb E Z\,.
	\end{equation*}
\end{lemma}

We now state the complex cumulant expansion formula that is a central ingredient of the proof. 
\begin{lemma}(Lemma 7.1,\cite{HK}) \label{lem:3.1}
	Let $h$ be a complex random variable with all its moments exist. The $(p,q)$-cumulant of $h$ is defined as
	$$
	\mathcal{C}_{p,q}(h)\deq (-i)^{p+q} \cdot \left(\frac{\partial^{p+q}}{\partial {s^p} \partial {t^q}} \log \bE e^{\mathrm{i}sh+\mathrm{i}t\bar{h}}\right) \bigg{|}_{s=t=0}\,.
	$$
	Let $f:\bC^2 \to \bC$ be a smooth function, and we denote its holomorphic  derivatives by
	$$
	f^{(p,q)}(z_1,z_2)\deq \frac{\partial^{p+q}}{\partial {z_1}^p \partial {z_2}^q} f(z_1,z_2)\,.
	$$ Then for any fixed $\ell \in \bN$, we have
	\begin{equation} \label{5.16}
	\bE f(h,\bar{h})h=\sum_{k=0}^{\ell}\sum_{\substack{p,q \ge 0,\\p+q=k}} \frac{1}{p!\,q!}\mathcal{C}_{p+1,q}(h)\bE f^{(p,q)}(h,\bar{h}) + R_{\ell+1}\,,
	\end{equation}
	given all integrals in \eqref{5.16} exists. Here $R_{\ell+1}$ is the remainder term depending on $f$ and $h$, and for any $t>0$, we have the estimate
	\begin{equation} \label{T3.4}
	\begin{aligned}
	R_{l+1}=&\ O(1)\cdot \bE \big|h^{\ell+2}\cdot\mathbf{1}_{\{|h|>t\}}\big|\cdot \max\limits_{p+q=\ell+1}\big\| f^{(p,q)}(z,\bar{z})\big\|_{\infty} \\
	&+O(1) \cdot \bE |h|^{\ell+2} \cdot \max\limits_{p+q=\ell+1}\big\| f^{(p,q)}(z,\bar{z})\cdot \mathbf{1}_{\{|z|\le t\}}\big\|_{\infty}\,.
	\end{aligned}
	\end{equation}
\end{lemma}

The bulk of the proof is performed on Wigner matrices satisfying a stronger condition than Definition \ref{def:Wigner} (iii) by having entries with finite moments of all order.
\begin{definition} \label{def:dWigner}
We consider the subset of Wigner matrices obtained from Definition \ref{def:Wigner} by replacing (iii) with
	\begin{itemize}
		\item [(iii)']
		For each $p \in \N$ there exists a constant $C_p$ such that $\E \abs{\sqrt{N} H_{ij}}^p \leq C_p$ for all $N,i,j$.
	\end{itemize}
\end{definition}

From now on we shall focus on Wigner matrices satisfying Definition. Relaxing condition (iii)' to (iii) using a Green function comparison argument is done in Section 6 of \cite{HK}, and one can see that the same method also works for this paper.

The following result gives bounds on the cumulants of the entries of $H$.
\begin{lemma} \label{Tlemh}
	If $H$ satisfies Definition \ref{def:dWigner} then for every $p,q \ge 0, p+q=k$, and $k \in \bb N$ we have
	\begin{equation*}
	\cal C_{p,q}(H_{ij})=O_{k}(N^{-k/2})
	\end{equation*}
	uniformly for all $i,j$, and $\cal C_{1,0}(H_{ij})=\cal C_{0,1}(H_{ij})=0$.
\end{lemma}
\begin{proof}
	 This follows by the homogeneity of the cumulants.
\end{proof}
We shall also use the following result in complex analysis, which can be viewed as an approximate Cauchy identity for almost analytic functions.
\begin{lemma} \label{lem5.1}
	Let $f\in C^{\infty}_c(\bb R)$. Fix $k \in \bb N+$, and let $\tilde{f}$ be the almost analytic extension of $f$ defined by
	\begin{equation} \label{5.1}
	\tilde{f}(x+\ii y)=f(x)+\sum_{j=1}^{k}\frac{1}{j!}(\ii y)^j f^{(j)}(x).
	\end{equation}
	Let $a>0$, and we denote $D_a=\{(x+\ii y):x \in \bb R, |y|\le a \}$. For any $\lambda \in \bb R$, we have
	\begin{equation*}
	f(\lambda)=\frac{\ii}{2\pi} \oint_{\partial D_a}\frac{\tilde f(z)}{\lambda-z}\, \dd z+\frac{1}{\pi} \int_{D_a} \frac{\partial_{\bar{z}}\tilde f(z)}{\lambda-z}\, \dd^2z\,.
	\end{equation*}
\end{lemma}
\begin{proof}
	Let $\varepsilon >0$ be small such that $B_{\varepsilon} (\lambda) \subset D_a$, where we use $B_{\varepsilon}(\lambda)$ to denote the ball centered at $\lambda$ with radius $\varepsilon$. The result follows by first use Green's formula for the function $\tilde f(z)/(\lambda-z)$ on $D_a\backslash B_{\varepsilon}(\lambda)$, and then let $\varepsilon \downarrow 0$.
\end{proof}

The following definition introduces a notion of a high-probability bound that is suited for our purposes. It was introduced (in a more general form) in \cite{EKYY4}.
\begin{definition}[Stochastic domination] \label{def:2.3} 
	Let $$X=\pb{X^{(N)}(u): N \in \bN, u \in U^{(N)}}\,,\qquad Y=\pb{Y^{(N)}(u): N \in \bN, u \in U^{(N)}}$$ be two families of nonnegative random variables, where $U^{(N)}$ is a possibly $N$-dependent parameter set. We say that $X$ is stochastically dominated by $Y$, uniformly in $u$, if for all (small) $\varepsilon>0$ and (large) $D>0$ we have
	\begin{equation} \label{2.10}
	\sup\limits_{u \in U^{(N)}}	\bP \left[ X^{(N)}(u) > N^{\varepsilon} Y^{(N)}(u) \right] \le N^{-D}
	\end{equation} 
	for large enough $N \ge N_0(\varepsilon,D)$. If $X$ is stochastically dominated by $Y$, uniformly in $u$, we use the notation $X \prec Y$. Moreover,
	if for some complex family $X$ we have $|X|\prec Y$ we also write $X=O_{\prec}(Y)$. (Note that for
	deterministic $X$ and $Y$, $X =O_\prec(Y)$ means $X= O_{\epsilon}(N^{\epsilon}Y)$ for any $\epsilon> 0$.)
\end{definition}

Next we recall the local semicircle law for Wigner matrices from \cite{EKYY4,EYY3}. For a recent survey of the local semicircle law, see \cite{BK16}, where the following version of the local semicircle law is stated.
\begin{theorem}[Local semicircle law] \label{refthm1}
	Let $H$ be a Wigner matrix satisfying Definition \ref{def:dWigner}, and define the spectral domain 
	$$ 
	{\bf S}  \deq  \{E+\mathrm{i}\eta: |E| \le 10, 0 <  \eta \le 10 \}\,.
	$$
	Then we have the bounds
	\begin{equation} \label{3.3}
	\max\limits_{i,j}|G_{ij}(z)-\delta_{ij}m(z)| \prec \sqrt{\frac{\im m(z)}{N\eta}}+
\frac{1}{N\eta}
	\end{equation} 
and
	\begin{equation} \label{3.4}
	|\underline{G(z)}-m(z)| \prec \frac{1}{N\eta}\,,
	\end{equation}
	uniformly in $z =   
	E+\mathrm{i}\eta \in {\bf S}$. 
\end{theorem}
	The following lemma is a preliminary estimate on $G$.
	\begin{lemma}[Lemma 4.4,\cite{HK}] \label{prop4.4}
		Let $H$ be a Wigner matrix satisfying Definition \ref{def:dWigner}. Fix $\alpha \in (0,1)$ and $E \in (-2,2)$. Let $G\equiv G(z)=(H-z)^{-1}$, where $z\deq E+\mathrm{i}\eta$ and $\eta\deq N^{-\alpha}$. For any fixed $k \in \bb N_{+}$ we have
		\begin{equation}  \label{Tk}
		\big|\big\langle \ul{G^{k}} \big\rangle \big|\prec N^{(k-1)\alpha-(1-\alpha)}
		\end{equation}
		as well as
		\begin{equation} \label{410}
		\big|\big(G^k\big)_{ij}\big|\prec
		\begin{cases}
		N^{(k-1)\alpha} & \txt{if } i = j
		\\
		N^{(k-1)\alpha-(1-\alpha)/2} & \txt{if } i \neq j\,,
		\end{cases}
		\end{equation}
		uniformly in $i,j$.
	\end{lemma}

In addition, we also have the following lemma, whose proof is given in Appendix \ref{TA}.
\begin{lemma} \label{Tlem3.8}
	We adopt the assumptions of Lemma \ref{prop4.4}. Let $F\deq (H-z')^{-1}$, where $z'\deq E'+\ii \eta$, and $E'\in (-2,2)$. We denote
		\begin{equation*}
		\cal G\deq \{G,G^{\intercal},\ol{G},G^{*},F,F^{\intercal},\ol{F},F^{*}\}\,.
		\end{equation*}
		For any fixed $k \in \bb N_{+}$, we have
	\begin{equation} \label{T3.144}
	\big|\big(G^{(1)}\cdots G^{(k)}\big)_{ij} \big| \prec N^{(k-1)\alpha}
	\end{equation}
	uniformly in $i,j$, and $G^{(1)},\dots,G^{(k)}\in \cal G$.
\end{lemma}

Lemmas \ref{prop4.4} and \ref{Tlem3.8} are very useful in estimating the expectations involving entries of $G$, in combination with the following elementary result about stochastic domination.

\begin{lemma} \label{prop_prec}
	\begin{enumerate}
		\item
		If $X_1 \prec Y_1$ and $X_2 \prec Y_2$ then $X_1 X_2 \prec Y_1 Y_2$.
		\item
		Suppose that $X$ is a nonnegative random variable satisfying $X \leq N^C$ and $X \prec \Phi$ for some deterministic $\Phi \geq N^{-C}$. Then $\E X \prec \Phi$.
	\end{enumerate}
\end{lemma}

We end this section with a result concerning $\bb E \ul{G}$ and $\bb E \ul{G^2}$, and the proof is postponed to Appendix \ref{TB}.
\begin{lemma} \label{Tlem3.9}
	We adopt the assumptions of Lemma \ref{prop4.4}. Let 
	\begin{equation} \label{T4.2}
	\chi \deq \frac{1}{2}\min\{\alpha,1-\alpha\}>0\,.
	\end{equation}
	We have
	\begin{equation} \label{T3.16}
	\bb E \ul{G^2} =O_{\prec}( N^{\alpha-\chi})
	\end{equation}
	as well as
	\begin{equation} \label{T3.17}
	\bb E \ul{G}-m =O_{\prec} (N^{\alpha-1-\chi}) \,.
	\end{equation}
\end{lemma}

\section{Outline of the proof} \label{Tsec3}

Let us denote $\ul{M}=N^{-1}\tr M$ for an $N\times N$ matrix $M$.

We prove the convergence of $\hat{Z}(f)$ in the sense of moments. As a starting point, we use an approximate Cauchy formula to reduce the problem to the case of the Green function $f(x)=(x-\ii)^{-1}$. The main work thus lies in computing $\bb E|\ul{G}-\bb E \ul{G} |^{2p}$, namely the centered moments of the Green function $G\deq(H-z)^{-1}$. Here $z\deq E+\ii \eta$, and $E,\eta$ are as in Theorem \ref{mainthm2}.  

In a second step, we use resolvent identity $zG=-I+GH$ to extract centered random variables $H_{ij}$ from the expression, and by cumulant expansion formula 
\begin{equation} \label{Tsimcumu}
\bb E f(h,\bar{h})h = \bb E |h|^2 \cdot \bb E \frac{\partial f(h,\bar{h})}{\partial \bar{h}}+\bb E h^2 \cdot \bb E \frac{\partial f(h,\bar{h})}{\partial h}+\cdots
\end{equation}
for centered random variables $h$, we get a preliminary self-consistent equation
\begin{equation} \label{T3.11}
\bb E|\ul{G}-\bb E \ul{G} |^{2p}=\frac{1}{z+2\bb E \ul{G}}\Big[-\frac{p}{N^2} \Big(\bb E \ul{G^{*2}G}+\sigma\bb E \ul{\ol{G}^2G}\Big) \bb E |\ul{G}-\bb E \ul{G}  |^{2p-2}+ \cal E_1\Big]\,,
\end{equation}
where the first term on RHS is the main term, and $\cal E_1$ denotes the error term.

In a third step, we analyse the terms in \eqref{T3.11}. For $z$ in our domain of interest, $(z+2\bb E \ul{G})^{-1}$ is bounded, and we only need to calculate 
\begin{equation} \label{++}
\bb E \ul{G^{*2}G}+\sigma\bb E \ul{\ol{G}^2G}
\end{equation}
to understand the main term in \eqref{T3.11}. The first term in \eqref{++} is easy, one just needs to apply the resolvent identity $G^{*}G=(G^{*}-G)/(\bar{z}-z)$ twice and use the local semicircle law Theorem \ref{refthm1}. For the second term in \eqref{++}, there are two extreme cases where $\sigma\bb E \ul{\ol{G}^2G}$ is only affected by one symmetry class, and the calculation is trivial. When $\sigma=1$, $H$ is in the same symmetric class as GOE, and it is real and symmetric. This makes $G$ symmetric, thus $\sigma \bb E \ul{\ol{G}^2G}=\bb E \ul{G^{*2}G}$. When $\sigma=0$, $H$ falls into the same symmetric class as GUE, and one simply has $\sigma\bb E \ul{\ol{G}^2G}=0$. In the intermediate case $\sigma\in (0,1)$,  $\sigma \bb E \ul{\ol{G}^2G}$ is the crucial term that reveals the mixed impact of both symmetry classes, and the computation is much more subtle. For example, a straight-forward calculation using
\begin{equation} \label{Teee}
z \bb E \ul{\ol{G}^2G}=-\bb E \ul{G^2} +\frac{1}{N}\sum_{i,j}\bb E (\ol{G}^2G)_{ij}H_{ji}
\end{equation}
and \eqref{Tsimcumu} leads to the self-consistent equation
\begin{equation}
\bb E \ul{\ol{G}^2G}= \frac{1}{z+\bb E \ul{G}+\sigma\bb E \ul{\ol{G}}} \big(-\bb E \ul{\ol{G}^2}+\cal E_2\big)
\end{equation}
where $\cal E_2$ denotes the error term. As $\sigma\to 1$, the true size of $\bb E \ul{\ol{G}^2G}$ is $O(\eta^{-2})$, and we are able to prove $\cal E_2=O(\eta^{-2}N^{-\epsilon})$ for some small $\epsilon>0$. However, by the local semicircle law, one has $(z+\bb E \ul{G}+\sigma\bb E \ul{\ol{G}})^{-1}\sim (1-\sigma+\eta)^{-1}$, thus the error term $(z+\bb E \ul{G}+\sigma\bb E \ul{\ol{G}})^{-1} \cal E_2$ blows up at size $O((1-\sigma+\eta)^{-1}\cdot \eta^{-2}N^{-\epsilon})=O(\eta^{-3}N^{-\epsilon})\gg \eta^{-2}$. To deal with the instability of the operator $(z+\bb E \ul{G}+\sigma\bb E \ul{\ol{G}})^{-1}$, we observe that a swap of summation indices gives
\begin{equation} \label{Tfff}
\bar{z} \bb E \ul{\ol{G}^2G}=-\bb E \ul{\ol{G}G}+\frac{1}{N}\sum_{i,j} \bb E \ol{H}_{ij} (\ol{G}^2G)_{ji}=-\bb E \ul{\ol{G}G}+\frac{1}{N} \sum_{i,j}\bb E  (\ol{G}^2G)_{ij} \ol{H}_{ji}\,,
\end{equation}
and this leaves the last summations in \eqref{Teee} and \eqref{Tfff} with exactly the same Green functions, while the $H$ entries are complex conjugates. Also, by \eqref{Tsimcumu} we have
\begin{equation} \label{IG}
\bb E f(h,\bar{h})(h-\bar{h}) = (\bb E |h|^2-\bb E \bar{h}^2) \cdot \bb E \frac{\partial f(h,\bar{h})}{\partial \bar{h}}+(\bb E h^2-\bb E |h|^2) \cdot \bb E \frac{\partial f(h,\bar{h})}{\partial h}+\cdots\,,
\end{equation}
and note that $\bb E |H_{ji}|^2-\bb E (\ol{H}_{ji})^2=1-\sigma$. Thus subtracting \eqref{Tfff} from \eqref{Teee} provides a crucial cancellation, together with \eqref{IG} yield
\begin{equation}
\bb E \ul{\ol{G}^2G}= (z-\bar{z}+(1-\sigma)(\bb E \ul{G}-\bb E \ul{\ol{G}}))^{-1} \big(-\bb E \ul{\ol{G}G}+(1-\sigma)\cal E_3\big)\,,
\end{equation} 
where $\cal E_3$ is an error term satisfying $\cal E_3=O(\eta^{-2}N^{-\epsilon})$, and $(z-\bar{z}+(1-\sigma)(\bb E \ul{G}-\bb E \ul{\ol{G}}))^{-1}\sim (1-\sigma+\eta))^{-1}$ by the local semicircle law. This implies
$(z-\bar{z}+(1-\sigma)(\bb E \ul{G}-\bb E \ul{\ol{G}}))^{-1}(1-\sigma)\cal E_3=O(\eta^{-2}N^{-\epsilon})
\ll \eta^{-2}$ as desired. The error term $\cal E_1$ in \eqref{T3.11} can be analysed in a similar fashion. 
	
\section{Proof of the main result} \label{sec4}
Let us abbreviate $f_{\eta}(x)\deq f\big(\frac{x-E}{\eta}\big)$, and denote $[\tr f_{\eta}(H)]\deq \tr f_{\eta}(H)- N \int_{-2}^{2} \varrho(x) f_{\eta}(x) \, \mathrm{d}x$. In this section we prove the following particular case of Theorem \ref{mainthm2}.
\begin{theorem} \label{prop:4.2}
	Let $H$ be a Wigner matrix satisfying Definition \ref{def:dWigner}. Let $\eta,E,\mu_*$ be as in Theorem \ref{mainthm2}. Let $f\in C_c^{\infty}(\bb R)$. Then
	\begin{equation*}
[  \tr f_{\eta}(H)] \overset{d}{\longrightarrow} \mathcal{N}\Big(0\,,\,V_{\mu_*}(f,f)\,\Big)
\end{equation*}
as $N \to \infty$, where $V_{\mu_*}(f,f)$ is defined as in \eqref{variance}. The convergence also holds in the sense of moments.
\end{theorem}

\begin{remark} \label{rmk5.2}
Our main work in the remaining of this paper will be showing
Theorem \ref{prop:4.2}. It is weaker than Theorem \ref{mainthm2} in three senses.
\begin{enumerate}
	\item[(A)] We only have 1-dimensional convergence of $[  \tr f_{\eta}(H)]$, instead of the joint convergence of $\hat{Z}(f)$.
	
	\item[(B)] The entries of $H$ satisfy stronger moment conditions, i.e. $\bb E|\sqrt{N}H_{ij}| \leq C_p$ for all $p$, instead of $\bb E|\sqrt{N}H_{ij}|^{4-2\delta_{ij}} \leq C$.
	
	\item[(C)] The test functions are in the class $C^{\infty}_c(\bb R)$ rather than in $C^{1,r,s}(\bb R)$.
\end{enumerate}
We make the above simplifications to focus on our main issue, which is going from $\sigma \in \{0,1\}$ to $\sigma \in [-1,1]$ in Definition \ref{def:Wigner}. The methods removing constrains (A) - (C) have been described in \cite{HK} for $\sigma \in \{0,1\}$. More precisely, in Section 5.5 of \cite{HK} we showed that the joint convergence can be proved in exactly the same way as the 1-dimensional convergence; in Section 6 of \cite{HK} we used Green function comparison argument to relax the moment conditions of $H$; in Section 5 of \cite{HK} we uses very careful cut-off of integration regions and estimates so that the error terms are small for test functions in the class $C^{1,r,s}(\bb R)$ (this was also explained just before Section 5.1 of \cite{HK}). One readily checks that these methods are completely insensitive to the value of $\sigma$, hence for conciseness and readability we only give a proof of Theorem \ref{prop:4.2} in this paper.
\end{remark}
 Recall that for a random variable $X$, $\langle X \rangle\deq X-\bE X$. Theorem \ref{prop:4.2} is a direct consequence of the following proposition.
\begin{proposition} \label{Tlem4.3}
	Let us adopt the conditions of Theorem \ref{prop:4.2}, and define $\mu\deq\sqrt{4-E^2}(1-\sigma)/\eta$. Let $\chi$ be as in \eqref{T4.2}. We have the following results.
	\begin{enumerate}
		\item For any fixed $n \ge 2 $
		\begin{equation} \label{wick}
		 \bE \langle \tr f_{\eta}(H)\rangle^n=(n-1)\,V_{\mu}(f,f) \cdot  \bE \langle \tr f_{\eta}(H)\rangle^{n-2}+O_{\prec}(N^{-\frac{\chi}{2n}})\,,
		\end{equation}
		where $V_{\mu}(f,f)$ is defined as in \eqref{variance}.
		\item The random variables $[  \tr f_{\eta}(H)]$ and $ \langle \tr f_{\eta}(H)\rangle$ are close in the sense that
		\begin{equation} \label{compare}
		[  \tr f_{\eta}(H)]-\langle \tr f_{\eta}(H)\rangle=\bE 	[  \tr f_{\eta}(H)]=O_{\prec}(N^{-{\chi}/2})\,.
		\end{equation}		 
	\end{enumerate} 
\end{proposition}
Assume Proposition \ref{Tlem4.3} holds. Then \eqref{T2.5}, \eqref{wick} and Wick's theorem imply
\begin{equation} \label{2.14.1}
\langle \tr f_{\eta}(H)\rangle \overset{d}{\longrightarrow} \mathcal{N}\Big(0\,,V_{\mu_*}(f,f)\Big)
\end{equation}
as $N \to \infty$. Note that the above result is proved in a stronger sense that we have convergence in moments. Theorem \ref{prop:4.2} then follows from (\ref{compare}).

\subsection{Proof of Proposition \ref{Tlem4.3} (i)}\label{Tsec4.1}
 Fix $n\ge 2$. In this section we calculate
 \begin{equation*}
 \bb E \langle \tr f_{\eta}(H) \rangle ^n\,.
 \end{equation*} 
 Let 
 \begin{equation} \label{T4.6}
 \gamma \deq\frac{\chi}{2n}\,, \ \ \ \omega \deq N^{-\gamma}\,,\ \ \mbox{ and }\ \ a\deq \eta\, \omega=N^{-(\alpha+\gamma)}\,,
 \end{equation}
 where $\chi$ is defined in \eqref{T4.2}. Note that our definition of $\gamma$ ensures $\alpha+\gamma \in (0,1)$. Applying Lemma \ref{lem5.1} with $a=\eta\,\omega$ and $k=n$ gives
  \begin{equation*}
 f_{\eta}(H)=\frac{\ii}{2\pi} \oint_{\partial D_a}\tilde f_{\eta}(z)G(z)\, \dd z+\frac{1}{\pi} \int_{D_a} \partial_{\bar{z}}\tilde f_{\eta}(z)G(z)\, \dd^2z\,,
 \end{equation*}
  thus 
  \begin{equation} \label{T4.7}
  \langle \tr f_{\eta}(H) \rangle=N\bigg(\frac{\ii}{2\pi} \oint_{\partial D_a}\tilde f_{\eta}(z)\langle \ul{G(z)}\rangle\, \dd z+\frac{1}{\pi} \int_{D_a} \partial_{\bar{z}}\tilde f_{\eta}(z)\langle \ul{G(z)}\rangle\, \dd^2z\bigg)\,.
  \end{equation} 
Let us write $z=x+\ii y$. By \eqref{3.4} and Lemma \ref{prop_prec} we know
\begin{equation*}
|\langle \ul{G(z)}\rangle|\leq |\ul{G(z)}-m|+|\bb E [\ul{G(z)}-m]|\prec \frac{1}{|Ny|}\,.
\end{equation*}
Note that $\partial_{\bar{z}}f_{\eta}(z)=\frac{1}{2k!}(\ii y)^kf^{(k+1)}_{\eta}(x)=\frac{1}{2n!}(\ii y)^nf^{(n+1)}_{\eta}(x)$ and $\int |f^{(n+1)}_{\eta}| \prec \eta^{-n}$, we have
\begin{equation} \label{T4.9}
\bigg|\frac{N}{\pi} \int_{D_a} \partial_{\bar{z}}\tilde f_{\eta}(z)\langle \ul{G}(z)\rangle\, \dd^2z\bigg|\prec \int_{D_a} \Big| y^kf^{(k+1)}_{\eta}(x) \frac{1}{y}\Big| \, \dd^2z \prec \frac{1}{\eta^n}\int_{-a}^a |y^{n-1}|\, \dd y \prec \frac{1}{\eta^n}\cdot(\eta\,\omega)^n=\omega^n\,.
\end{equation}
Note that $\partial D_a=\{z \in \bb C: \im z= \pm a\}$ and $\oint _{\partial D_a} |\tilde{f}_{\eta}(z)|\, \dd z=O(\eta)$, thus
\begin{equation} \label{T4.10}
\bigg|\frac{\ii N}{2\pi}\oint_{\partial D_a}\tilde f_{\eta}(z)\langle \ul{G(z)}\rangle\, \dd z\,\bigg| \prec  N\oint_{\partial D_a}\bigg|\tilde f_{\eta}(z)\frac{1}{Na}\bigg|\, \dd z \prec \frac{\eta}{a}=\frac{1}{\omega}\,.
\end{equation}
Then \eqref{T4.7}, \eqref{T4.9} and \eqref{T4.10} imply 
\begin{equation} \label{T4.11}
\bb E \langle \tr f_{\eta}(H) \rangle^n = \Big(\frac{\ii N}{2\pi}\Big)^n \oint_{(\partial D_a)^n} \tilde{f}_{\eta}(z_1)\cdots \tilde{f}_{\eta}(z_n)\, \bb E \langle \ul{G(z_1)} \rangle \cdots \langle \ul{G(z_n)} \rangle\, \dd z_1\cdots \dd z_n + O_{\prec}(\omega)\,,
\end{equation}
where we use $\oint_{(\partial D_a)^n}$ to denote $n$ multiples of $\int_{\partial D_a}$. The core of our proof lies in computing the first term on RHS of \eqref{T4.11}. Since $f$ has compact support and $\eta=N^{-\alpha}$, we know $\tilde{f}_{\eta}(z)=0$ whenever $\re z \notin (-2+\tau/2, 2-\tau/2)$, where $\tau>0$ is given in the beginning of Theorem \ref{mainthm2}. Thus by writing $z_k=x_k+\ii y_k$ for $k=1,2,...,n$, when only consider when $(z_1,...,z_n)$ lies in the region
\begin{equation} \label{T414}
A \deq \{(z_1,...,z_n)\in \bC^n: x_k \in [-2+\tau/2,2-\tau/2], y_k=\pm a\,, \mbox{for }k=1,2,...,n\}\,.
\end{equation}
Let us abbreviate 
$$
Q_m\deq \langle \ul{G(z_1)} \rangle \cdots \langle \ul{G(z_m)} \rangle\ \mbox{ and }\ Q_m^{(k)}\deq Q_m/\langle \ul{G(z_k)} \rangle 
$$
for all $1 \le k \le m \le n$. Let $\zeta_i\deq \bE |\sqrt{N}H_{ii}|^2$, and note that (iii) in Definition \ref{def:dWigner} implies $\max_i \zeta_i \le C$. By resolvent identity $zG=GH-I$ we have
\begin{equation} \label{T4.15}
z_n\bb E Q_n=\bb E \big[\langle Q_{n-1}\rangle \ul{G(z_n)H}\,\big]=\frac{1}{N}\sum_{i,j} \bb E \big[ \langle Q_{n-1} \rangle  G_{ij}(z_n)H_{ji}\big]\,.
\end{equation}
Since $H$ is complex hermitian, for any differentiable $f=f(H)$ we set
\begin{equation} \label{diff2}
\frac{\partial }{\partial H_{ij}}f(H)\deq \frac{\mathrm{d}}{\mathrm{d}t}\Big{|}_{t=0} f\pb{H+t\,{\Delta}^{(ij)}}\,,
\end{equation}
where ${\Delta}^{(ij)}$ denotes the matrix whose entries are zero everywhere except at the site $(i,j)$ where it is one: ${\Delta}^{(ij)}_{kl} =\delta_{ik}\delta_{jl}$. Applying Lemma \ref{5.16} with $f=f_{ij}(H)=\langle Q_{n-1} \rangle G_{ij}(z_n)$ and $h=H_{ji}$ to the above, we get
\begin{equation} \label{3.12} 
\begin{aligned} 
z\bE  Q_n &= \frac{1}{N^2} \sum\limits_{i,j} \bE\frac{\partial( \langle Q_{n-1} \rangle \cdot G_{ij}(z_n))}{\partial H_{ij}}+ \frac{\sigma}{N^2} \sum\limits_{i,j} \bE\frac{\partial( \langle Q_{n-1} \rangle \cdot G_{ij}(z_n))}{\partial H_{ji}} + K + L \\
&\eqd  (a)+(b)+K+ L\,,
\end{aligned} 
\end{equation}
where 
\begin{equation} \label{T3.13}
K=\frac{1}{N^2} \sum\limits_{i} \bE\frac{\partial( \langle Q_{n-1} \rangle  G_{ii}(z_n))}{\partial H_{ii}}(\zeta_i-1-\sigma)\,,
\end{equation}
and
\begin{equation} \label{T3.14} 
L=N^{-1} \cdot \sum\limits_{i,j}\Bigg[\sum_{k=2}^{\ell}\sum_{\substack{p,q \ge 0,\\p+q=k}}\frac{1}{p!\,q!}\mathcal{C}_{p+1,q}(H_{ji})\bE\frac{\partial^{p+q}( \langle Q_{n-1} \rangle G_{ij}(z_n))}{ \partial {H_{ji}}^p \partial {H_{ij}}^q}+R_{\ell+1}^{(ji)}\Bigg].  
\end{equation}
Here $\ell$ is a fixed positive integer to be chosen later, and $R_{\ell+1}^{(ji)}$ is a remainder term defined analogously to $R_{\ell+1}$ in (\ref{T3.4}). More precisely, we have the bound
\begin{equation} \label{tau}
\begin{aligned}	
R_{\ell+1}^{(ji)}&=O(1) \cdot \bE \big|H_{ji}^{\ell+2}\mathbf{1}_{\{|H_{ji}|>t\}}\big|\cdot  \max_{p+q=\ell+1}\big\| \partial_{ji}^{p}\partial_{ij}^qf_{ij}\big(H\big)\big\|_{\infty}\\&\ \ +O(1) \cdot \bE \big|H_{ji}^{\ell+2}\big| \cdot   \bE \max_{p+q=\ell+1} \big\|\partial_{ji}^{p}\partial_{ij}^qf_{ij}\big(\hat{H}+z\Delta^{(ji)}+\bar{z}\Delta^{(ij)}\big)\cdot \b1_{|x| \le t}\big\|_{\infty}
\end{aligned}
\end{equation}
for any $t>0$, where we define $\hat{H}\deq H-H_{ji}\Delta^{(ji)}- H_{ij}\Delta^{(ij)}$, so that the matrix $\hat{H}$ has zero entries at the positions $(i,j)$ and $(j,i)$, and abbreviate $\partial_{ij} \deq \frac{\partial}{\partial H_{ij}}$.
Note that for $G=(H-z)^{-1}$ we have
\begin{equation} \label{3.15}
\frac{\partial G_{ij}}{\partial H_{kl}}=-G_{ik}G_{lj}\,,
\end{equation}
thus
\begin{equation*} 
\begin{aligned}
(a)&=- \frac{1}{N^2} \sum_{i,j} \bE\bigg[ \sum_{k=1}^{n-1} \frac{1}{N}Q_{n-1}^{(k)} (G^2(z_k))_{ji}G_{ij}(z_n)+\langle Q_{n-1} \rangle G_{ii}(z_n)G_{jj}(z_n)\bigg]\\
&= -\frac{1}{N^2}\sum\limits_{k=1}^{n-1} \bE Q_{n-1}^{(k)} \ul{{G}^2(z_k) G(z_n)}-\bb E \langle Q_{n-1}\rangle \ul{G(z_n)}^2 \\[0.5em]
&=  -\frac{1}{N^2}\sum\limits_{k=1}^{n-1} \bE Q_{n-1}^{(k)} \ul{{G}^2(z_k) G(z_n)} - \bE Q_{n} \langle \ul{G(z_n)} \rangle -2\bE Q_{n}\bE \underline{G(z_n)} + \bE Q_{n-1}\bE \langle \underline{G(z_n)} \rangle^2\,,
\end{aligned} 
\end{equation*} 
where in the last step we used Lemma \ref{Tlemc}.
Similarly, we have
\begin{equation*} 
(b)=-\frac{\sigma}{N^2}\sum\limits_{k=1}^{n-1} \bE Q_{n-1}^{(k)} \ul{{G}^{2}(z_k) G^{\intercal} (z_n)} -\frac{\sigma}{N}\bb E Q_{n-1} \langle \ul{G(z_n)G^{\intercal}(z_n)}\rangle\,.
\end{equation*}
Altogether we obtain
\begin{equation} \label{T4.20}
\begin{aligned}
	\bE Q_n=&\ \frac{1}{T_n} \bE Q_n \langle \ul{G(z_n)} \rangle- \frac{1}{T_n}\bE Q_{n-1} \bE \langle \ul{G(z_n)} \rangle^2+\frac{\sigma}{NT_n}\bE Q_{n-1} \langle \ul{G(z_n)G^{\intercal}(z_n)}\rangle\\
	& -\frac{K}{T_n}-\frac{L}{T_n}
	+\frac{1}{N^2T_n}\sum\limits_{k=1}^{n-1} \bE Q_{n-1}^{(k)} \ul{{G}^2(z_k) G(z_n)}+\frac{\sigma}{N^2T_n}\sum\limits_{k=1}^{n-1} \bE Q_{n-1}^{(k)} \ul{{G}^{2}(z_k) G^{\intercal} (z_n)}
\end{aligned}
\end{equation}
where $T_n\deq -z_n-2\bE \underline{G(z_n)}$. From (\ref{T3.3}), (\ref{3.4}), and Lemma \ref{prop_prec} it is easy to see that 
\begin{equation} \label{T}
\Big|\frac{1}{T_n}\Big|=O(1)\,,
\end{equation}
and the implicit constant depends only on the distance to the spectral edge $\tau$ defined in Theorem \ref{mainthm2}.

\begin{remark} \label{Trmk4.4}
Before looking at the terms in \eqref{T4.20}, let us compare \eqref{T4.20} to its special case when $H$ is real and symmetric, whose analysis was done in \cite{HK}. When $H$ is real and symmetric, we have $\sigma=1$, and $G^{\intercal}(z_n)=G(z_n)$. Thus \eqref{T4.20} becomes 
\begin{equation} \label{T4.23}
\begin{aligned}
\bE Q_n=&\ \frac{1}{T_n} \bE Q_n \langle \ul{G(z_n)} \rangle- \frac{1}{T_n}\bE Q_{n-1} \bE \langle \ul{G(z_n)} \rangle^2+\frac{1}{NT_n}\bE Q_{n-1} \langle \ul{G^2(z_n)}\rangle\\
& -\frac{K}{T_n}-\frac{L}{T_n}
+\frac{2}{N^2T_n}\sum\limits_{k=1}^{n-1} \bE Q_{n-1}^{(k)} \ul{{G}^2(z_k) G(z_n)}\,,
\end{aligned}
\end{equation}
which coincides with (5.21) in \cite{HK}. The symmetry of $G$ simplifies the analysis of many terms, for example the last term on the RHS of \eqref{T4.23} is the leading term, and it can be treated with the resolvent identity
\begin{equation*}
G^2(z_k)G(z_n)=-\frac{G(z_k)-G(z_n)}{(z_k-z_n)^2}+\frac{G^2(z_k)}{z_k-z_n}\,,
\end{equation*}
which allows us to proceed the computation easily; when $G$ is not symmetric, we have the last two terms on the RHS of \eqref{T4.20} instead, and the analysis of $G^2(z_k)G^{\intercal}(z_n)$ is more subtle. The asymmetry of $G$ is dealt by Lemma \ref{Tlem4.5} below, and we prospond the proof to Section \ref{Tsec5}. On the other hand, the estimates of $L$ and $K$ in \eqref{T4.23} can easily be applied to \eqref{T4.20}, and we summarize these results in Lemma \ref{Tlem4.2} below.
\end{remark}

\begin{lemma} \label{Tlem4.2}
	We extend the definition of $\chi$ in \eqref{T4.2} to a function $\chi(\cdot): (0,1) \to \bR$ such that
	\begin{equation} \label{functionc_0}
	\chi(x)\deq \frac{1}{2}\min\{x,1-x\}\,.
	\end{equation}
	Let $K,L$ be as in \eqref{T4.20}. We have the following results.
	\begin{enumerate}
		\item $|L|=O_{\prec}(N^{n(\alpha+\gamma-1)-\chi(\alpha+\gamma)})$.
		\item $|K|=O_{\prec}(N^{n(\alpha+\gamma-1)-(\alpha+\gamma)})$.
	\end{enumerate}
\end{lemma}
\begin{proof} 
	(i) For $k \ge 2$, we define
	\begin{equation} \label{T4.27}
	J_k\deq {N^{-(k+3)/2}} \cdot \sum\limits_{i,j}\sum_{\substack{p,q \ge 0,\\p+q=k}}\bigg|\bE\frac{\partial^{p+q}( \langle Q_{n-1} \rangle G_{ij}(z_n))}{ \partial {H_{ji}}^p \partial {H_{ij}}^q}\bigg|\,, 
	\end{equation}
    and Lemma \ref{Tlemh} shows 
	\begin{equation} \label{T4.28}
	|L |\leq \sum_{k=2}^{\ell} O(J_k) +\frac{1}{N} \sum_{i,j} |R^{(ji)}_{\ell+1}|\,. 
	\end{equation}
	The estimate of $J_k$ is similar to the proof of Lemma 4.6 (ii) in \cite{HK}. More precisely, after applying the differentials in \eqref{T4.27} carefully, we can use Lemmas \ref{prop4.4}-\ref{prop_prec} to show that
	\begin{equation} \label{T4.29}
	J_k =O_{\prec}(N^{n(\alpha+\gamma-1)-\chi(\alpha+\gamma)})
	\end{equation}
	for any fixed $k \ge 2$. Also, one can following the routine verification of Lemma 4.6 (i) in \cite{HK} to show there exists some $L \in  \bb N$ such that $\frac{1}{N} \sum_{i,j} |R^{(ji)}_{\ell+1}|=O(N^{-(n+1)})$ for all $\ell \ge L$. Together with \eqref{T4.28} and \eqref{T4.29} we get the desired result.
	
	(ii) The proof is analogue to that of Lemma 4.7 in \cite{HK}. By $\max_i \zeta_i \leq C$ we see that
	\begin{equation*}
	|K| = O\Big(\frac{1}{N^2}\Big) \cdot \sum\limits_{i} \bE\Big|\frac{\partial( \langle Q_{n-1} \rangle  G_{ii}(z_n))}{\partial H_{ii}}\Big|\,,
	\end{equation*}
	and the proof follows by Lemmas \ref{prop4.4} -- \ref{prop_prec}.
\end{proof}

\begin{lemma} \label{Tlem4.5}
	Let $H$ be a Wigner matrix satisfying Definition \ref{def:dWigner}. Fix $\alpha \in (0,1)$ and let $E,E' \in (-2,2)$. Let $G\deq  (H-z)^{-1}$ and $F\deq (H-z')^{-1}$, where $z\deq E+\mathrm{i}\eta$, $z'\deq E'+\ii \eta$, and $\eta\deq N^{-\alpha}$. Let $\chi$ be as in \eqref{T4.2}. We have the following results.
	\begin{enumerate}
		\item $\bb E \ul{G^2\ol{F}}=-\nu^{-2} (m(z)-m(\bar{z}'))+O_{\prec}(N^{2\alpha-\chi})$, where $\nu\equiv\nu(z,\bar{z}')$ is defined in \eqref{Tnu} below.
		\item $\langle \ul{G\ol{F}} \rangle\prec N^{2\alpha-1}$ and  $ \langle \ul{G^2\ol{F}} \rangle \prec N^{3\alpha-1}$\,.
		\item $\bb E \ul{GF^{\intercal}} \prec N^{\alpha-\chi}$ and $\langle\ul{GF^{\intercal}}\rangle \prec N^{2\alpha-1}$. 
		\item $\bb E \ul{G^2F^{\intercal}} \prec N^{2\alpha-\chi}$ and $\langle\ul{G^2F^{\intercal}}\rangle \prec N^{3\alpha-1}$. 
		\item $\ul{G^2F}\prec N^{2\alpha-\chi}$\,.
		\item $\bb E \ul{G^2F^*}=-(z-\bar{z}')^{-2}(m(z)-m(\bar{z}'))+O_{\prec}(N^{2\alpha-\chi})$ and $\langle \ul{G^2F^*}\rangle \prec N^{3\alpha-1}$.
	\end{enumerate}
\end{lemma}
\begin{proof}
	See Section \ref{Tsec5}.
\end{proof}

Now let us turn to \eqref{T4.20}. Lemma \ref{prop4.4} and \eqref{T} imply
\begin{equation} \label{T4.30}
\frac{1}{T_n} \bE Q_n \langle \ul{G(z_n)} \rangle- \frac{1}{T_n}\bE Q_{n-1} \bE \langle \ul{G(z_n)} \rangle^2=O_{\prec}(N^{n(\alpha+\gamma-1)-\chi(\alpha+\gamma)})\,,
\end{equation}
where $\chi(\cdot)$ was defined in \eqref{functionc_0}. Also, by Lemmas \ref{prop4.4}, \ref{Tlem4.5} (iii) and \eqref{T} we have
\begin{equation} \label{T4.31}
\frac{\sigma}{NT_n}\bE Q_{n-1} \langle \ul{G(z_n)G^{\intercal}(z_n)}\rangle=O_{\prec}(N^{n(\alpha+\gamma-1)-\chi(\alpha+\gamma)})\,.
\end{equation}
Inserting Lemma \ref{Tlem4.2} and \eqref{T4.30}-\eqref{T4.31} into \eqref{T4.20} yields
\begin{equation} \label{o}
\bE Q_n=\frac{1}{N^2T_n}\sum\limits_{k=1}^{n-1} \bE Q_{n-1}^{(k)} \ul{{G}^2(z_k) G(z_n)}+\frac{\sigma}{N^2T_n}\sum\limits_{k=1}^{n-1} \bE Q_{n-1}^{(k)} \ul{{G}^{2}(z_k) G^{\intercal} (z_n)}+O_{\prec}(N^{n(\alpha+\gamma-1)-\chi(\alpha+\gamma)})
\end{equation}	
uniformly for all $(z_1,\dots,z_n)\in A$, where $A$ was defined in \eqref{T414}. Note that by the definitions of $\alpha$, $\gamma$ and $\chi$ we have $\chi(\alpha+\gamma)\ge \chi -\gamma/2=\chi-\chi/(4n)\ge 7\chi/8$. Since we have the simple estimate 
\begin{equation} \label{simple}
\oint _{\partial D_a} |\tilde{f}_{\eta}(z)|\, \dd z=O(\eta)\,,
\end{equation}
we know from \eqref{T4.11} and \eqref{o} that
\begin{equation} \label{hehe}
\begin{aligned}
&\ \bb E \langle \tr f_{\eta}(H)\rangle^n\\ = &\ \Big(\frac{\ii}{2\pi}\Big)^n \frac{N^{n-2}}{T_n}\sum\limits_{k=1}^{n-1}\oint_{(\partial D_a)^n} \tilde{f}_{\eta}(z_1)\cdots \tilde{f}_{\eta}(z_n)\bE Q_{n-1}^{(k)} \big(\ul{{G}^2(z_k) G(z_n)}+\sigma\, \ul{{G}^2(z_k) G^{\intercal}(z_n)}\big)\mathrm{d} z_1\cdots \mathrm{d} z_n+O_{\prec}(\omega)\,,
\end{aligned}
\end{equation}
where in the estimate of the error term we implicitly used
\begin{equation} \label{Tim}
n\gamma-\chi(\alpha+\gamma) =\chi/2-\chi(\alpha+\gamma)\leq -3\chi/8 \leq -\gamma\,.
\end{equation} 
By symmetry, it suffices to fix $k \in \{1,2,\dots,n-1\}$, and consider the integral of 
\begin{equation} \label{430}
F_{kn}\deq \Big(\frac{\ii}{2\pi}\Big)^n \frac{N^{n-2}}{T_n}\tilde{f}_{\eta}(z_1)\cdots \tilde{f}_{\eta}(z_n)\bE Q_{n-1}^{(k)} \big(\ul{{G}^2(z_k) G(z_n)}+\sigma\, \ul{{G}^2(z_k) G^{\intercal}(z_n)}\big)\,.
\end{equation}

The next Lemma together with \eqref{hehe} will conclude the proof of Proposition \ref{Tlem4.3} (i).

\begin{lemma} \label{lem4.6}
	Let $	F_{kn}$ be as in (\ref{430}), and recall the definition of $\omega$ from \eqref{T4.6}. We have 
		\begin{equation} \label{T4.37}
		\oint_{(\partial D)^n} \b 1_{(y_k=y_n=a)}F_{nk} =O_{\prec}(\omega)\,,
		\end{equation}
		and
		\begin{multline} \label{T4.38}
		\oint_{(\partial D)^n} \b 1_{(y_k=a,y_n=-a)}F_{nk}\\ =\frac{1}{8\pi^2}\int (f(u)-f(v))^2\Big(\frac{1}{(u-v)^2}+\frac{1}{(u-v+\ii \mu)^2}\Big)\,\dd u\,\dd v \cdot\bb E \langle \tr f_{\eta}(H)\rangle^{n-2} +O_{\prec}(\omega)\,.
		\end{multline}
		Similarly,
		\begin{equation} \label{T4.39}
		\oint_{(\partial D)^n} \b 1_{(y_k=y_n=-a)}F_{nk} =O_{\prec}(\omega)\,,
		\end{equation}		
and     \begin{multline} \label{T4.40}
		\oint_{(\partial D)^n} \b 1_{(y_k=-a,y_n=a)}F_{nk}\\ =\frac{1}{8\pi^2}\int (f(u)-f(v))^2\Big(\frac{1}{(u-v)^2}+\frac{1}{(u-v-\ii \mu)^2}\Big)\,\dd u\,\dd v\cdot \bb E \langle \tr f_{\eta}(H)\rangle^{n-2}+O_{\prec}(\omega)\,.
		\end{multline}
\end{lemma}
\begin{proof} We only prove \eqref{T4.37} and \eqref{T4.38}, the proof of \eqref{T4.39} and \eqref{T4.40} follows in the same fashion.
	
	(i) Let us first look at \eqref{T4.37}. By Lemma \ref{Tlem4.5} (iv), (v) we see that
	\begin{equation*}
	\big|\ul{{G}^2(z_k) G(z_n)}\big|\,,\big|\ul{{G}^2(z_k) G^{\intercal}(z_n)}\big| \prec N^{2(\alpha+\gamma)-\chi(\alpha+\gamma)}
	\end{equation*}
	uniformly in  $A\cap\{(z_1,...,z_n)\in \bb C: y_k=y_n=a\}$, where $\chi(\cdot)$ was defined in \eqref{functionc_0}. Also, Lemma \ref{prop4.4} shows $|Q_{n-1}^{(k)}| \prec N^{(n-2)(\alpha+\gamma-1)}$ uniformly in  $A\cap\{(z_1,...,z_n)\in \bb C: y_k=y_n=a\}$. Thus \eqref{T} shows
	\begin{equation*}
	\Big(\frac{\ii}{2\pi}\Big)^n\frac{N^{n-2}}{T_n}\bE Q_{n-1}^{(k)} \big(\ul{{G}^2(z_k) G(z_n)}+\sigma\, \ul{{G}^2(z_k) G^{\intercal}(z_n)}\big)=O_{\prec}(N^{n(\alpha+\gamma)-\chi(\alpha+\gamma)})
	\end{equation*}
	uniformly in $A\cap\{(z_1,...,z_n)\in \bb C: y_k=y_n=a\}$. Together with \eqref{simple} and \eqref{Tim} we have
	\begin{equation*}
	\oint_{(\partial D)^n} \b 1_{(y_k=y_n=a)}F_{nk}=O_{\prec}(N^{-n\alpha}\cdot N^{n(\alpha+\gamma)-\chi(\alpha+\gamma)})=O_{\prec}(\omega).
	\end{equation*}
	
	(ii) Now we look at \eqref{T4.38}. By Lemma \ref{Tlem4.5} (i), (ii), and (vi), we see that
	\begin{equation*}
	\ul{{G}^2(z_k) G(z_n)}+\sigma\, \ul{{G}^2(z_k) G^{\intercal}(z_n)}=-((z_k-z_n)^{-2}+\sigma\nu(z_k,z_n)^{-2})(m(z_k)-m(z_n))+O(N^{2(\alpha+\gamma)-\chi(\alpha+\gamma)})
	\end{equation*}
	uniformly in $A\cap\{(z_1,...,z_n)\in \bb C: y_k=a,y_n=-a\}$. Also, Lemma \ref{prop4.4} shows $|Q_{n-1}^{(k)}| \prec N^{(n-2)(\alpha+\gamma-1)}$ uniformly in  $A\cap\{(z_1,...,z_n)\in \bb C: y_k=a,y_n=-a\}$. Thus Lemma \ref{Tk} and \eqref{T} imply 
	\begin{multline*}
	\frac{N^{n-2}}{T_n}\bE Q_{n-1}^{(k)} \big(\ul{{G}^2(z_k) G(z_n)}+\sigma\, \ul{{G}^2(z_k) G^{\intercal}(z_n)}\big)\\
	=(z_n+2m(z_n))^{-1}((z_k-z_n)^{-2}+\sigma\nu(z_k,z_n)^{-2})(m(z_k)-m(z_n))N^{n-2} \bb E Q_{n-1}^{(k)}+O_{\prec}(N^{n(\alpha+\gamma)-\chi(\alpha+\gamma)})\\
	\eqd \phi(z_k,z_n) N^{n-2} \bb E Q_{n-1}^{(k)}+O_{\prec}(N^{n(\alpha+\gamma)-\chi(\alpha+\gamma)})\,.
	\end{multline*}
	Thus \eqref{simple} and \eqref{Tim} imply
	\begin{equation} \label{T4.47}
	\begin{aligned}
	&\ \ \ \oint_{(\partial D)^n} \b 1_{(y_k=a,y_n=-a)}F_{nk}\\
	 &= \Big(\frac{\ii}{2\pi}\Big)^n\int_{(\partial D_a)^n} \b 1_{(y_k=a,y_n=-a)} \tilde{f}_{\eta}(z_1)\cdots \tilde{f}_{\eta}(z_n)\,\phi(z_k,z_n) N^{n-2} \bb E Q_{n-1}^{(k)} \dd z_1\cdots \dd z_n+O_{\prec}(\omega)\\
	&= -\frac{1}{4\pi^2}\oint_{(\partial D_a)^2} \b 1_{(y_k=a,y_n=-a)} \tilde{f}_{\eta}(z_k)\tilde{f}_{\eta}(z_n)\phi(z_k,z_n)\, \dd z_k \dd z_n \cdot \Big(\frac{\ii N}{2\pi}\Big)^{n-2} \oint_{(\partial D_a)^{n-2}} \hat{F}_{kn} +O_{\prec}(\omega)\,,
	\end{aligned}
	\end{equation}
	where
	\begin{equation*}
	\hat{F}_{kn}\deq\tilde{f}_{\eta}(z_1)\cdots \tilde{f}_{\eta}(z_{n-1})/\tilde{f}_{\eta}(z_k)\, \bb E Q_{n-1}^{(k)}\,.
	\end{equation*}
	Same as in \eqref{T4.11}, one can show that
	\begin{equation*}
	\oint_{(\partial D_a)^{n-2}} \hat{F}_{kn} =\bb E \langle \tr f_{\eta}(H) \rangle^{n-2}+O_{\prec}(\omega)\,,
	\end{equation*}
	together with \eqref{T4.47} we have
	\begin{equation} \label{T4.49}
	\begin{aligned}
	&\oint_{(\partial D)^n} \b 1_{(y_k=a,y_n=-a)}F_{nk}\\
	=&-\frac{1}{4\pi^2}\oint_{(\partial D_a)^2} \b 1_{(y_k=a,y_n=-a)} \tilde{f}_{\eta}(z_k)\tilde{f}_{\eta}(z_n)\phi(z_k,z_n)\, \dd z_k \dd z_n \cdot \bb E \langle \tr f_{\eta}(H) \rangle^{n-2}+O_{\prec}(\omega)\,.\\
	=&-\frac{1}{4\pi^2}\int_{+\infty}^{-\infty}\int_{-\infty}^{+\infty}\tilde{f}_{\eta}(x_k+\ii a)\tilde{f}_{\eta}(x_n-\ii a)\phi(x_k+\ii a,x_n-\ii a)\, \dd x_k \dd x_n \cdot \bb E \langle \tr f_{\eta}(H) \rangle^{n-2}+O_{\prec}(\omega)\,.
	\end{aligned}
	\end{equation}	
	Note that in for $z \in \supp \tilde f_{\eta}$, we have $|\re z-E|=O{(\eta)}$, together with $E \in [-2+\tau,2-\tau]$ one easily show that
    \begin{equation*}
    \phi(x_k+\ii a,x_n-\ii a)=-(x_k-x_n+2\ii  a)^{-2}-\sigma[(x_k-x_n)+2\ii a+\ii \sqrt{4-E^2}(1-\sigma)]^{-2}+O(N^{2(\alpha+\gamma)-\chi(\alpha+\gamma)\wedge \alpha})
    \end{equation*}
    whenever $x_k+\ii a,x_n-\ii a\in \supp \tilde{f}_{\eta}$. Let us denote 
    \begin{equation*}
    \varphi(z_k,z_n)\deq -(z_k-z_n)^{-2}-\sigma(z_k-z_n+\ii \sqrt{4-E^2}(1-\sigma))^{-2}\,.
    \end{equation*}
	 By $\int |\tilde{f}_{\eta}(x+\ii a) |\dd x=O(\eta)$ and $\int \varphi(x_k+\ii a,x_n-\ii a) \,\dd x_k =\int \varphi(x_k+\ii a,x_n-\ii a) \,\dd x_n=0$ we have
	\begin{multline} \label{T4.55}
	\int_{+\infty}^{-\infty}\int_{-\infty}^{+\infty}\tilde{f}_{\eta}(x_k+\ii a)\tilde{f}_{\eta}(x_n-\ii a)\phi(x_k+\ii a,x_n-\ii a)\, \dd x_k \dd x_n\\
	=\int_{+\infty}^{-\infty}\int_{-\infty}^{+\infty}\tilde{f}_{\eta}(x_k+\ii a)\tilde{f}_{\eta}(x_n-\ii a)\varphi(x_k+\ii a,x_n-\ii a)\, \dd x_k \dd x_n+O(\omega)\\
	=\frac{1}{2}\int_{-\infty}^{+\infty}\int_{-\infty}^{+\infty}\big(\tilde{f}_{\eta}(x_k+\ii a)-\tilde{f}_{\eta}(x_n-\ii a)\big)^2\varphi(x_k+\ii a,x_n-\ii a)\, \dd x_k \dd x_n+O(\omega)\,.
	\end{multline} 
	Now an elementary estimate shows 
	\begin{multline*}
	\frac{1}{2}\int_{-\infty}^{+\infty}\int_{-\infty}^{+\infty}\big(\tilde{f}_{\eta}(x_k+\ii a)-\tilde{f}_{\eta}(x_n-\ii a)\big)^2\varphi(x_k+\ii a,x_n-\ii a)\, \dd x_k \dd x_n\\
	=-\frac{1}{2}\int (f(u)-f(v))^2\Big(\frac{1}{(u-v)^2}+\frac{\sigma}{(u-v+\ii \sqrt{4-E^2}(1-\sigma)/\eta)^2}\Big)\,\dd u\,\dd v+O(\omega)\\
	=-\frac{1}{2}\int (f(u)-f(v))^2\Big(\frac{1}{(u-v)^2}+\frac{1}{(u-v+\ii \mu)^2}\Big)\,\dd u\,\dd v+O(\omega)\,,
	\end{multline*} 
	together with \eqref{T4.49} and \eqref{T4.55} we compete the proof.
\end{proof}

\subsection{Proof of Proposition \ref{Tlem4.3} (ii)} \label{sec4.4}
Let 
\begin{equation*}
\gamma\deq \chi/4\,,\ \ \omega\deq N^{-\gamma}\,,\ \ \mbox{ and }\ \ a\deq \eta\,\omega=N^{-(\alpha+\gamma)}\,,
\end{equation*}
where $\chi$ is as in \eqref{T4.2}. Note that $\alpha+\gamma \in (0,1)$. Applying Lemma \ref{lem5.1} with $a=\eta\,\omega$ and $k=2$ gives
\begin{equation} \label{T4.58}
\bb E[\tr f_{\eta}(H)]=\frac{\ii N}{2\pi} \oint_{\partial D_a}\tilde f_{\eta}(z)(\bb E\ul{G(z)}-m(z))\, \dd z+\frac{N}{\pi} \int_{D_a} \partial_{\bar{z}}\tilde f_{\eta}(z)(\bb E\ul{G(z)}-m(z))\, \dd^2z\,.
\end{equation}
Note that \eqref{T3.17} implies
\begin{equation*}
\bb E\ul{G(z)}-m(z)=O_{\prec}(N^{\alpha+\gamma-1-\chi(\alpha+\gamma)})
\end{equation*}
uniformly for $z \in \partial D_a \cap \supp \tilde{f}_{\eta}$.
Thus
\begin{equation} \label{T4.60}
\bigg|\frac{\ii N}{2\pi} \oint_{\partial D_a}\tilde f_{\eta}(z)(\bb E\ul{G(z)}-m(z))\, \dd z\bigg|=O_{\prec}(N^{\gamma-\chi(\alpha+\gamma)})=O_{\prec}(N^{-\chi/2})\,.
\end{equation}
By Theorem \ref{refthm1} we have
\begin{equation} \label{T4.61}
\bigg|\frac{N}{\pi} \int_{D_a} \partial_{\bar{z}}\tilde f_{\eta}(z)(\bb E\ul{G(z)}-m(z))\, \dd^2z\bigg|\prec \int_{D_a} \Big| y^2f^{(3)}_{\eta}(x) \frac{1}{y}\Big| \, \dd^2z \prec \omega^2\,.
\end{equation}
Inserting \eqref{T4.60} and \eqref{T4.61} into \eqref{T4.58} completes the proof.

\section{Proof of Lemma \ref{Tlem4.5}} \label{Tsec5}
In this section we prove Lemma \ref{Tlem4.5}. As a preparation, we have the following analogue of Lemma \ref{lem:3.1}, which will be used throughout this section. This lemma provides the key cancellation that allows us to get a stable self-consistent equation for the terms in Lemma \ref{Tlem4.5} (i) and (ii). See Remark  \ref{Trmk5.3} below for a detailed discussion.

\begin{lemma} \label{Tlem5.1}
Let us adopt the assumptions of Lemma \ref{lem:3.1}. Then for any fixed $\ell \in \bb N$, we have
\begin{equation} \label{T5.166}
\bE f(h,\bar{h})(h-\bar{h})=\sum_{k=0}^{\ell}\sum_{\substack{p,q \ge 0,\\p+q=k}} \frac{1}{p!\,q!}\big(\mathcal{C}_{p+1,q}(h)-\cal{C}_{p,q+1}(h)\big)\bE f^{(p,q)}(h,\bar{h}) + \tilde{R}_{\ell+1}\,,
\end{equation}
given all integrals in \eqref{T5.166} exists. Here $\tilde{R}_{\ell+1}$ is the remainder term depending on $f$ and $h$, and for any $t>0$, $\tilde{R}_{\ell+1}$ satisfy the same bound as $R_{\ell+1}$ in \eqref{T3.4}.
\end{lemma}
\begin{proof}
	By Lemma \ref{lem:3.1} we have
	\begin{equation*}
	\begin{aligned}
	\bb E f(h,\bar{h})\bar{h}=\sum_{k=0}^{\ell}\sum_{\substack{p,q \ge 0,\\p+q=k}}\frac{1}{p!\,q!}\mathcal{C}_{p+1,q}(\bar{h})\bE f^{(q,p)}(h,\bar{h}) + \hat{R}_{\ell+1}=\sum_{k=0}^{\ell}\sum_{\substack{p,q \ge 0,\\p+q=k}} \frac{1}{p!\,q!}\cal{C}_{p,q+1}(h)\bE f^{(p,q)}(h,\bar{h})+ \hat{R}_{\ell+1}\,,
	\end{aligned}
	\end{equation*}
	where $\hat{R}_{\ell+1}$ is the remainder term satisfying the same bound as $R_{\ell+1}$ in \eqref{T3.4}. The proof then follows from \eqref{5.16}.
\end{proof}
When $\sigma$ is close to $1$, Lemma \ref{Tlem5.1} provides a crucial cancellation for the cumulant expansion. It is summarized in the following lemma.
\begin{lemma} \label{Tlem:5.2}
If $H$ satisfies Definition \ref{def:dWigner} then
\begin{equation} \label{T5.1}
\cal C_{2,0}(H_{ij})-\cal C_{1,1}(H_{ij})=(\sigma-1)/N\,, \ \ \ \cal C_{1,0}(H_{ij})-\cal C_{0,2}(H_{ij})=(1-\sigma)/N
\end{equation}
for all $i \ne j$. For $p,p'\in \{0,1,2,3\}$ we have
\begin{equation} \label{T5.2}
\cal C_{p,3-p}(H_{ij})-\cal C_{p',3-p'}(H_{ij})=O\big(N^{-3/2}\sqrt{1-\sigma}\big)
\end{equation}
uniformly for all $i\ne j$.

\begin{proof}
By $\cal C_{2,0}(H_{ij})=\cal C_{0,2}(H_{ij})=\sigma/N$ and $\cal C_{1,1}(H_{ij})=1/N$ we trivially obtain \eqref{T5.1}.

Let us look at \eqref{T5.2}. Let us write $H_{ij}=a+\ii b$ for some $i\ne j$. Definition \ref{def:dWigner} (iii)' ensures
	\begin{equation} \label{T5.3}
	\bb E |a|^k=O_k(N^{-k/2}) \ \ \mbox{ and }\ \ \bb E |b|^k=O_k(N^{-k/2}) 
	\end{equation} 
	for all fixed $k \in \bb N$. By \eqref{T5.1} we see that
	\begin{equation*}
	(1-\sigma)/N=\cal C_{1,1}({H_{ij}})-\cal C_{2,0}(H_{ij})=\bb E (a^2+b^2)-\bb E(a^2-b^2+2\ii ab)=2\bb E b^2-2\ii \bb E ab\,,
	\end{equation*}
	thus 
	\begin{equation} \label{T5.5}
	\bb E \,b^2 =\frac{1-\sigma}{2N}\,.
	\end{equation}
	Then \eqref{T5.2} follows from  
	$$\cal C_{p,3-p}(H_{ij})-\cal C_{p',3-p'}(H_{ij})=\bb E (a+\ii b)^p(a-\ii b)^{3-p}-\bb E (a+\ii b)^{p'}(a-\ii b)^{3-p'}\,,$$
	together with \eqref{T5.3}, \eqref{T5.5}, and H\"{o}lder's inequality. 
\end{proof}
\end{lemma}
We are now ready to prove Lemma \ref{Tlem4.5}. We first give the proof of parts (i) and (ii).
\begin{proof}[Proof of Lemma \ref{Tlem4.5} (i)-(ii)]	
	(i) By resolvent identity $z G= HG -I$ we have
	\begin{equation} \label{T5.6}
	z \bb E \ul{G^2\ol{F}} = \bb E \ul{HG^2 \ol{F}}-\bb E \ul{G\ol{F}}=\frac{1}{N} \sum_{i,j} \bb E H_{ij}(G^2\ol{F})_{ji}-\bb E \ul{G\ol{F}}\,.
	\end{equation}
	Similarly, by $\bar{z}'\ol{F}=\ol{F}\ol{H}-I=\ol{F}H^{\intercal}-I$ we have
	\begin{equation} \label{T5.7}
	\bar{z}' \bb E \ul{G^2\ol{F}} = \bb E \ul{G^2 \ol{F}H^{\intercal}}-\bb E \ul{G^2}=\frac{1}{N} \sum_{i,j} \bb E (G^2\ol{F})_{ij}H^{\intercal}_{ji}-\bb E \ul{G^2}=\frac{1}{N} \sum_{i,j} \bb E (G^2\ol{F})_{ji}H_{ji}-\bb E \ul{G^2}\,,
	\end{equation}
	where in the last step we swap the summation indexes $i$ and $j$. By \eqref{T5.6} and \eqref{T5.7} we have
	\begin{equation} \label{T5.8}
	(z-\bar{z}')\bb E\ul{G^2\ol{F}}=\frac{1}{N}\sum_{i\ne j} \bb E (H_{ij}-H_{ji})(G^2\ol{F})_{ji}+\bb E \ul{G^2}-\bb E \ul{G\ol{F}}\,.
	\end{equation}
	We apply Lemma \ref{T5.166} to the first term on RHS of \eqref{T5.8} with $h=H_{ij}$ and $f=f_{ij}=(G^2\ol{F})_{ji}$ and get
	\begin{equation} \label{T5.9}
	\begin{aligned}
	(z_1-\bar{z}')\bb E\ul{G^2\ol{F}}=&\,\frac{1}{N}\sum_{i\ne j}\big(\cal C_{2,0}(H_{ij})-\cal C_{1,1}(H_{ij})\big)\bb E \frac{\partial (G^2\ol{F})_{ji}}{\partial H_{ij}}+\frac{1}{N}\sum_{i\ne j}\big(\cal C_{1,1}(H_{ij})-\cal C_{0,2}(H_{ij})\big)\bb E \frac{\partial (G^2\ol{F})_{ji}}{\partial H_{ji}}\\
	&+L^{(1)}+\bb E \ul{G^2}-\bb E \ul{G\ol{F}}\\
	=&\,\frac{\sigma-1}{N^2}\sum_{i,j}\bb E \frac{\partial (G^2\ol{F})_{ji}}{\partial H_{ij}}+\frac{1-\sigma}{N^2}\sum_{i,j}\bb E \frac{\partial (G^2\ol{F})_{ji}}{\partial H_{ji}}+L^{(1)}+\bb E \ul{G^2}-\bb E \ul{G\ol{F}}\,,
	\end{aligned}
	\end{equation}
	where 
	\begin{equation} \label{TL1}
	L^{(1)}\deq \frac{1}{N}\sum_{i \ne j}\bigg[\sum_{k=2}^{\ell}\sum_{\substack{p,q \ge 0,\\p+q=k}} \frac{1}{p!\,q!}\big(\mathcal{C}_{p+1,q}(H_{ij})-\cal{C}_{p,q+1}(H_{ij})\big)\bE \frac{\partial (G^2\ol{F})_{ji}}{\partial H_{ij}^p \partial H_{ji}^q}+R_{\ell+1}^{(1,ij)}\bigg]\,,
	\end{equation}
	and in the second step of \eqref{T5.9} we used the trivial fact
	\begin{equation*}
	\frac{\sigma-1}{N^2}\sum_{i}\bb E \frac{\partial (G^2\ol{F})_{ii}}{\partial H_{ii}}+\frac{1-\sigma}{N^2}\sum_{i}\bb E \frac{\partial (G^2\ol{F})_{ii}}{\partial H_{ii}}=0\,.
	\end{equation*}
	Here $\ell$ is a fixed positive integer to be chosen later, and $R_{\ell+1}^{(1,ji)}$ is a remainder term defined analogously to $R_{\ell+1}$ in (\ref{T3.4}). Note that we have
	\begin{equation} \label{T3.155}
	\frac{\partial G_{ij}}{\partial H_{kl}}=-G_{ik}G_{lj}\ \ \mbox{ and }\ \ \frac{\partial \ol{F}_{ij}}{\partial H_{kl}}=-\ol{F}_{il}\ol{F}_{kj}\,,
	\end{equation}
	thus
	\begin{equation*}
	\begin{aligned}
	\frac{1-\sigma}{N^2}\sum_{i,j}\bb E \frac{\partial (G^2\ol{F})_{ji}}{\partial H_{ji}}&=\frac{1-\sigma}{N^2}\sum_{i,j}\bb E [-G_{jj}(G^2\ol{F})_{ii}- (G^2)_{jj}(G\ol{F})_{ii}-(G^2\ol{F})_{ji}\ol{F}_{ji}]\\
	&=(\sigma-1)\, \Big(\bb E \ul{G} \,\bb E \ul{G^2\ol{F}}+ \bb E \langle\ul{G}\rangle \ul{G^2\ol{F}}+\bb E \ul{G^2}\, \bb E \ul{G\ol{F}} + \bb E \langle \ul{G^2} \rangle \ul{G\ol{F}}+\frac{1}{N}\bb E \ul{G^2\ol{F}F^*}  \Big)\,.
	\end{aligned}
	\end{equation*}
	By Lemmas \ref{prop4.4}-\ref{Tlem3.9} we see that
	\begin{equation*}
	\bb E \langle\ul{G}\rangle \ul{G^2\ol{F}} =O_{\prec} (N^{3\alpha-1})\,, \ \ \bb E \ul{G^2}\, \bb E \ul{G\ol{F}}=O_{\prec}(N^{2\alpha-\chi})\,,\ \ \bb E \langle \ul{G^2} \rangle \ul{G\ol{F}}=O_{\prec} (N^{3\alpha-1})\,,
	\end{equation*}
	and $N^{-1}\bb E \ul{G^2\ol{F}F^*}=O_{\prec} (N^{3\alpha-1})$.
	Thus
	\begin{equation} \label{T5.16}
	\frac{1-\sigma}{N^2}\sum_{i,j}\bb E \frac{\partial (G^2\ol{F})_{ji}}{\partial H_{ji}}=(\sigma-1)\, \bb E \ul{G} \,\bb E \ul{G^2\ol{F}}+O_{\prec}\big((1-\sigma)N^{2\alpha-\chi}\big)\,.
	\end{equation}
	Similarly, we have
	\begin{equation} \label{T5.17}
	\begin{aligned}
	\frac{\sigma-1}{N^2}\sum_{i,j}\bb E \frac{\partial (G^2\ol{F})_{ji}}{\partial H_{ij}}&=(1-\sigma)\Big( \bb E \ul{G^{\intercal}G^2\ol{F}}+ \bb E\ul{G^{\intercal 2}G\ol{F}}+\bb E \ul{\ol{F}}\, \bb E \ul{G^2\ol{F}}+\bb E \langle \ul{\ol{F}} \rangle \ul{G^2\ol{F}}\Big)\\
	&=(1-\sigma)\bb E \ul{\ol{F}}\, \bb E \ul{G^2\ol{F}}+O_{\prec}\big((1-\sigma)N^{2\alpha-\chi}\big)\,.
	\end{aligned}
	\end{equation}
	For $k \ge 2$, we denote
	\begin{equation*}
	J_k^{(1)}={N^{-(k+3)/2}}\sum_{i \ne j}\sum_{\substack{p,q \ge 0,\\p+q=k}} \bigg|\bE \frac{\partial^k (G^2\ol{F})_{ji}}{\partial H_{ij}^p \partial H_{ji}^q} \bigg|\,,
	\end{equation*}
	and Lemmas \ref{Tlemh} and \ref{Tlem:5.2} imply
	\begin{equation*}
	|L^{(1)}| \leq \sqrt{1-\sigma}\,O(J^{(1)}_2)+\sum_{k=3}^{\ell} O(J^{(1)}_k) +\frac{1}{N} \big|R^{(1,ji)}_{\ell+1}\big|\,.
	\end{equation*}
	Analogue to the proof of Lemma \ref{Tlem4.2}, we can use Lemmas \ref{prop4.4}-\ref{prop_prec} to show that
	\begin{equation*}
	J^{(1)}_2=O_{\prec}(N^{2\alpha-\chi-1/2})\,,\ \ \  J^{(1)}_k=O_{\prec}(N^{2\alpha-1})=O_{\prec}(N^{\alpha-\chi})
	\end{equation*}
	for $k \ge 3$, and there exists fixed $L>0$ such that $ \big|R^{(1,ji)}_{\ell+1}\big|=O(N^{-1})$ for all $\ell \ge L$. Thus
	\begin{equation} \label{T5.21}
|L^{(1)}| =O_{\prec}(\sqrt{1-\sigma}\, N^{2\alpha-\chi-1/2})+O_{\prec}(N^{\alpha-\chi})\,.
	\end{equation}
	Also, Lemma \ref{Tlem3.9} shows
	\begin{equation} \label{T5.22}
	\bb E \ul{G^2}=O_{\prec}(N^{\alpha-\chi})\,.
	\end{equation}
	Inserting \eqref{T5.16} -- \eqref{T5.22} into \eqref{T5.9} gives
	\begin{equation*}
	\big(z-\bar{z}'+(1-\sigma)(\bb E \ul{G}-\bb E \ul{\ol{F}})\big) \bb E\ul{G^2\ol{F}} = -\bb E\ul{G\ol{F}}+O_{\prec}\big((1-\sigma)N^{2\alpha-\chi}+\sqrt{1-\sigma}\, N^{2\alpha-\chi-1/2}+N^{\alpha-\chi}\big)\,.
	\end{equation*}
	By \eqref{3.4} we see that 
	\begin{equation} \label{T5.24}
	\bb E \ul{G} = m(z)+ O_{\prec}(N^{\alpha-1})\,, \ \ \mbox{ and }\ \ \bb E \ul{\ol{F}}=m(\bar{z}')+O_{\prec}(N^{\alpha-1})\,,
	\end{equation} 	
	together with the bound $\bb E \ul{G^2\ol{F}}=O_{\prec}(N^{2\alpha})$ from Lemma \ref{Tlem3.8} and our assumption $\eta =N^{-\alpha}\ge N^{-1}$, we have
	\begin{multline} \label{T517}
	\big(z-\bar{z}'+(1-\sigma)(m(z)-m(\bar{z}'))\big)\bb E\ul{G^2\ol{F}}\\=-\bb E\ul{G\ol{F}}+O_{\prec}\big((1-\sigma)N^{3\alpha-1}+(1-\sigma)N^{2\alpha-\chi}+\sqrt{1-\sigma}\, N^{2\alpha-\chi-1/2}+N^{\alpha-\chi}\big)\\
	=-\bb E\ul{G\ol{F}}+O_{\prec}\big((1-\sigma+\eta)N^{2\alpha-\chi}\big)\,.
	\end{multline}
	We abbreviate
	\begin{equation} \label{Tnu}
	\nu\equiv\nu(z,\bar{z}')\deq \big(z-\bar{z}'+(1-\sigma)(m(z)-m(\bar{z}'))\big)\,.
	\end{equation}
	Trivially
    \[
	\nu= \frac{1}{2}[(E-E')(1+\sigma)+\ii \,(4\eta+(\sqrt{4-E^2}+\sqrt{4-E'^2})(1-\sigma))]+O((1-\sigma)\eta)\,,
	\]
	and by $E,E'\in [-2+\tau,2-\tau]$ and taking the imaginary part of $\nu$ we see that
	\begin{equation} \label{T5.28}
	\nu^{-1}=O\Big(\frac{1}{1-\sigma+\eta}\Big)\,.
	\end{equation}
	Thus we arrive at
	\begin{equation*}
	\bb E\ul{G^2\ol{F}}=-\nu^{-1} \bb E \ul{G\ol{F}}+O_{\prec}(N^{2\alpha-\chi})\,.
	\end{equation*}
	Similarly, we can show that
	\begin{equation*}
	\bb E \ul{G\ol{F}}=\nu^{-1} (\bb E \ul{G}-\bb E \ul{\ol{F}})+O_{\prec}(N^{\alpha-\chi})\,,
	\end{equation*}
	thus 
	\begin{equation*}
	\bb E\ul{G^2\ol{F}}=-\nu^{-2} (\bb E \ul{G}-\bb E \ul{\ol{F}})+O_{\prec}(N^{2\alpha-\chi})
	\end{equation*}
    and the proof follows from \eqref{T5.24} and \eqref{T5.28}.
	
	(ii) Let us look at $P \deq \langle \ul{G\ol{F}}\rangle$. Fix $p \in \bb N_{+}$, by resolvent identity and Lemma \ref{Tlem5.1} we have
	\begin{multline} \label{T5.32}
	(z-\bar{z}')\bb E |P|^{2p}=\frac{\sigma-1}{N^2}\sum_{i,j}\bb E \frac{\partial [(G\ol{F})_{ji}\langle P^{p-1}\ol{P}^{p}\rangle]}{\partial H_{ij}}+\frac{1-\sigma}{N^2}\sum_{i,j}\bb E \frac{\partial[ (G\ol{F})_{ji}\langle P^{p-1}\ol{P}^p\rangle]}{\partial H_{ji}}+L^{(2)}\\
	+\bb E \langle\ul{G}\rangle P^{p-1}\ol{P}^p-\bb E \langle\ul{\ol{F}}\rangle P^{p-1} \ol{P}^{p}\,,
	\end{multline}
	where $L^{(2)}$ is defined similar as $L^{(1)}$ in \eqref{TL1}. One readily checks that
	\begin{multline*}
	\frac{1-\sigma}{N^2}\sum_{i,j}\bb E \frac{\partial[ (G\ol{F})_{ji}\langle P^{p-1}\ol{P}^p\rangle]}{\partial H_{ji}}\\=(\sigma-1) \bb E\Big[ \Big(\ul{G}\cdot\ul{G\ol{F}}+\ul{G^2}\cdot\ul{\ol{F}}+\frac{1}{N}\ul{G\ol{F}F^*}\Big)\langle P^{p-1}\ol{P}^p\rangle\Big]
	+\frac{(p-1)(\sigma-1)}{N^2}\bb E \big[(\ul{G\ol{F}G\ol{F}G}+\ul{G\ol{F}F^*G^{\intercal}F^*})P^{p-2}\ol{P}^p\big]\\+\frac{p(\sigma-1)}{N^2} \bb E  \big[(\ul{G\ol{F}G^*F^{\intercal}G^*}+\ul{G\ol{F}F\ol{G}F})|P|^{2p-2}\big]\,,
	\end{multline*}
	and by Lemma \ref{Tlem3.8} we see that the last two terms on the above equation are bounded by
	\begin{equation*}
	O_{\prec}((1-\sigma)N^{4\alpha-2}) \cdot \bb E|P|^{2p-2}\,.
	\end{equation*}
	Lemmas \ref{Tlemc}, \ref{prop4.4} and \ref{Tlem3.8} imply 
	\begin{multline*}
	(\sigma-1)\bb E \big[\ul{G} \cdot \ul{G\ol{F}} \langle P^{p-1} \ol{P}^{p}\rangle\big]\\
	=(\sigma-1) \big(\bb E \ul{G} \,\bb E |P|^{2p}+\bb E \ul{G\ol{F}} \,\bb E \big[\langle \ul{G}\rangle P^{p-1}\ol{P}^p\big]+\bb E \big[\langle \ul{G}\rangle \langle \ul{G\ol{F}}\rangle P^{p-1}\ol{P}^{p}\big]-\bb E \langle \ul{G}\rangle \langle \ul{G\ol{F}}\rangle \bb E P^{p-1}\ol{P}^p\big)\\
	=(\sigma-1)\bb E \ul{G} \,\bb E |P|^{2p}+O_{\prec}\big((1-\sigma)N^{2\alpha-1}\big)\cdot \bb E |P|^{2p-1}\,,
	\end{multline*}
	and similarly
	\begin{equation*}
	(\sigma-1)\bb E\Big[\Big(\ul{G^2}\cdot\ul{\ol{F}}+\frac{1}{N}\ul{G\ol{F}F^*}\Big)\langle P^{p-1}\ol{P}^p\rangle\Big]=O_{\prec}\big((1-\sigma)N^{2\alpha-1}\big)\cdot \bb E |P|^{2p-1}\,.
	\end{equation*}
	Thus
	\begin{multline} \label{T5.37}
	\frac{1-\sigma}{N^2}\sum_{i,j}\bb E \frac{\partial[ (G\ol{F})_{ji}\langle P^{p-1}\ol{P}^p\rangle]}{\partial H_{ji}}\\=(\sigma-1)\bb E \ul{G} \,\bb E |P|^{2p}+O_{\prec}\big((1-\sigma)N^{2\alpha-1}\big)\cdot \bb E |P|^{2p-1}+O_{\prec}((1-\sigma)N^{4\alpha-2}) \cdot \bb E|P|^{2p-2}\,.
	\end{multline}
	Similarly,
	\begin{multline} \label{TTT}
	\frac{\sigma-1}{N^2}\sum_{i,j}\bb E \frac{\partial [(G\ol{F})_{ji}\langle P^{p-1}\ol{P}^{p}\rangle]}{\partial H_{ij}}\\=(1-\sigma)\bb E \ul{\ol{F}} \,\bb E |P|^{2p}+O_{\prec}\big((1-\sigma)N^{2\alpha-1}\big)\cdot \bb E |P|^{2p-1}+O_{\prec}((1-\sigma)N^{4\alpha-2}) \cdot \bb E|P|^{2p-2}\,.
	\end{multline}
	Similar as in part (i), we can use Lemmas \ref{prop4.4}-\ref{prop_prec} and \ref{Tlem:5.2} to show that
	\begin{equation} \label{Tkkk}
	\begin{aligned}
	|L^{(2)}| &=O_{\prec}(\sqrt{1-\sigma}N^{3\alpha/2-1}) \cdot\bb E |P|^{2p-1}+O_{\prec}(\sqrt{1-\sigma}N^{7\alpha/2-2}) \cdot\bb E |P|^{2p-2}+\sum_{n=1}^{2p} O_{\prec}(N^{(2n-1)\alpha-n}) \cdot \bb E |P|^{2p-n}\\
	&=O_{\prec}((1-\sigma+\eta) N^{2\alpha-1}) \cdot\bb E |P|^{2p-1}+O_{\prec}((1-\sigma+\eta) N^{4\alpha-2}) \cdot\bb E |P|^{2p-2}+\sum_{n=1}^{2p} O_{\prec}(\eta N^{n(2\alpha-1)}) \cdot \bb E |P|^{2p-n}\,,
	\end{aligned}
	\end{equation} 
	and Lemma \ref{prop4.4} gives
	\begin{equation} \label{T5.40}
	\bb E \langle \ul{G} \rangle P^{p-1}\ol{P}^p =O_{\prec}(N^{\alpha-1}) \cdot \bb E |P|^{2p-1}\,,\ \ \mbox{ and }\ \ \bb E \langle \ul{\ol{F}}\rangle P^{p-1}\ol{P}^p=O_{\prec}(N^{\alpha-1}) \cdot \bb E |P|^{2p-1}\,.
	\end{equation}
	By inserting \eqref{T5.37}-\eqref{T5.40} into \eqref{T5.32}, we have
	\begin{equation*}
	\big(z-\bar{z}'+(1-\sigma)(\bb E \ul{G}-\bb E \ul{\ol{F}})\big) \bb E|P|^{2p}=(1-\sigma+\eta)\sum_{n=1}^{2p} O_{\prec}(N^{n(2\alpha-1)}) \cdot \bb E |P|^{2p-n}\,.
	\end{equation*}
	By \eqref{T5.24}, \eqref{Tnu} and \eqref{T5.28} we have $$\big(z-\bar{z}'+(1-\sigma)(\bb E \ul{G}-\bb E \ul{\ol{F}})\big)^{-1}=(\nu+O_{\prec}((1-\sigma)N^{\alpha-1}))^{-1}=O\Big(\frac{1}{1-\sigma+\eta}\Big)\,,$$
	thus by Jensen's Inequality we arrive at
	\begin{equation*}
	\bb E |P|^{2p} =\sum_{n=1}^{2p} O_{\prec}(N^{n(2\alpha-1)}) \cdot \bb E |P|^{2p-n}\leq \sum_{n=1}^{2p} O_{\prec}(N^{n(2\alpha-1)}) \cdot \big(\bb E |P|^{2p}\big)^{(2p-n)/2p}\,,
	\end{equation*}
	which implies $\langle \ul{G\ol{F}} \rangle =P =O_{\prec} (N^{2\alpha-1})$.
	
	The estimate of $\langle \ul{G^2\ol{F}}\rangle $ is done in the similar fashion, and we omit the details.	
\end{proof}

Now we give the proof of Lemma \ref{Tlem4.5} (iii)-(vi), and it follows a quite standard approach developed in \cite{HK2,HK} using Lemma \ref{lem:3.1} instead of Lemma \ref{Tlem5.1}.

\begin{proof} [Proof of Lemma \ref{Tlem4.5} (iii)-(v)]
	(iii) Let us look at $\bb E \ul{GF^{\intercal}}$. The estimate is similar to that of $\bb E \ul{G^2}$ in Appendix \ref{TB}.  By the resolvent identity and Lemma \ref{lem:3.1}, we arrive at
	\begin{equation} \label{T5.43}
	\bE \underline{GF^{\intercal}} = \frac{1}{\tilde{T}} \Big (\bb E \underline{F^{\intercal}} +  \bE \langle \underline{G} \rangle \langle \underline{GF^{\intercal}} \rangle+\frac{1}{N}\bb E  \ul{GF^{\intercal}F}+{\sigma} \bb E \langle\ul{F^{\intercal}}\rangle\langle\ul{GF^{\intercal}}\rangle + \frac{\sigma}{N} \bE \underline{GF^{\intercal}G^{\intercal}} -K^{(3)}-L^{(3)}\Big)\,,
	\end{equation} 
	where $\tilde T \deq -z-\bb E \ul{G}-\sigma\bb E \ul{F^{\intercal}}$,
	and $K^{(3)}$ and $L^{(3)}$ are defined similarly as $K^{(5)}$ and $L^{(5)}$ in \eqref{eqn: 2.59} below. Note that by Theorem \ref{refthm1}, 
	\begin{equation} \label{T5.433}
	\im \tilde T = -\sqrt{4-E^2}-\sigma\sqrt{4-E'^2}+O_{\prec}(N^{-\chi})\,,
	\end{equation} 
	thus $|{\tilde T}^{-1}|=O(1)$. Also, we see that the estimate of \eqref{eqn: 2.59} below also works for our current context, thus we can use Lemmas \ref{prop4.4}-\ref{prop_prec}, and show that every term on the RHS of \eqref{T5.43} is bounded by $O_{\prec}(N^{\alpha-\chi})$. This implies $\bb E \ul{GF^{\intercal}}=O_{\prec}(N^{\alpha-\chi})$.
	
	Let us look at $\langle \ul{GF^{\intercal}} \rangle\eqd Q$. Fix $p \in \bb N_{+}$. Similar as \eqref{T4.20}, we can use resolvent identity and Lemma \ref{lem:3.1} to show that
	\begin{multline} \label{T5.45}
	\bE |Q|^{2p}= \frac{1}{\tilde T}\Big(\bb E Q^{p-1}\ol{Q}^p \langle \ul{\ol{F}}\rangle+ \bE  |Q|^{2p}\langle \ul{G} \rangle+\bb E \ul{GF^{\intercal}}\bb E Q^{p-1}\ol{Q}^p\langle \ul{G}\rangle - \bE Q^{p-1}\ol{Q}^{p}\bb E \langle \ul{G}\rangle \langle 
	\ul{GF^{\intercal}}\rangle\\
	 +\frac{1}{N}\bb E Q^{p-1}\ol{Q}^{p}\ul{GF^{\intercal}F}+\sigma\bE  |Q|^{2p}\langle \ul{F^{\intercal}} \rangle+\sigma\bb E \ul{GF^{\intercal}}\bb E Q^{p-1}\ol{Q}^p\langle \ul{F^{\intercal}}\rangle - \sigma\bE Q^{p-1}\ol{Q}^{p}\bb E \langle \ul{F^{\intercal}}\rangle \langle 
	\ul{GF^{\intercal}}\rangle\\
	+\frac{1}{N}\bb E Q^{p-1}\ol{Q}^{p}\ul{GF^{\intercal}G^{\intercal}}
	+\frac{(p-1)}{N^2} \bE Q^{p-2}\ol{Q}^p \big(\ul{GF^{\intercal}GF^{\intercal}G}+\ul{GF^{\intercal}FG^{\intercal}F}\big) \\
	 +\frac{p}{N^2} \bE |Q|^{2p-2} \big(\ul{GF^{\intercal}G^{*}\ol{F}G^{*}}+\ul{GF^{\intercal}F^{*}G^{\intercal}F^*}\big)+\frac{(p-1)\sigma}{N^2} \bE Q^{p-2}\ol{Q}^p \big(\ul{GF^{\intercal}G^{\intercal} FG^{\intercal}}+\ul{GF^{\intercal 2}GF^{\intercal}}\big)\\
	 +\frac{p\sigma}{N^2} \bE |Q|^{2p-2} \big(\ul{GF^{\intercal}\ol{G}F^*\ol{G}}+\ul{GF^{\intercal}\ol{G}F^*\ol{G}}\big)-K^{(4)}-L^{(4)}\Big)\,,
	\end{multline}	
	where $K^{(4)}$ and $L^{(4)}$ are defined similar as $K$ and $L$ in \eqref{3.12}. By Lemmas \ref{prop4.4}-\ref{prop_prec}, one readily deduces from \eqref{T5.45} that
	\begin{equation*}
	\bb E|Q|^{2p}= \sum_{n=1}^{2p} O_{\prec}(N^{n(2\alpha-1)}) \cdot \big(\bb E |Q|^{2p}\big)^{(2p-n)/2p}\,,
	\end{equation*} 
	which implies $Q=\langle \ul{GF^{\intercal}} \rangle =O_{\prec}(N^{2\alpha-1})$.
	
	(iv) The proof is similar to that of part (iii), and we omit the details.
	
	(v) The proof is similar to those of part (iv). More precisely, it consists of showing $\bb E \ul{G^2F} \prec N^{2\alpha-\chi}$ and $\langle\ul{G^2F}\rangle \prec N^{3\alpha-1}$, and one readily checks that it makes no difference whether we have $F$ of $F^{\intercal}$ in the expression. We omit the details.
	
	(vi) The proof follows from the resolvent identity
	\begin{equation*}
	G(z)G(z')=\frac{G(z)-G(z')}{z-z'}
	\end{equation*}
	and Lemmas \ref{prop4.4} and \ref{Tlem3.9}.
\end{proof}
\begin{remark} \label{Trmk5.3}
	The method estimating $\bb E \ul{GF^{\intercal}}$ is not effective in analysing $\bb E \ul{G\ol{F}}$. More precisely, as an analogue of \eqref{T5.43}, we will get
	\begin{equation} \label{T5.47}
	\bE \underline{G\ol{F}} = \frac{1}{\hat{T}} \Big (\bb E \underline{\ol{F}} +  \bE \langle \underline{G} \rangle \langle \underline{G\ol{F}} \rangle+\frac{1}{N}\bb E  \ul{G\ol{F}F^*}+{\sigma} \bb E \langle\ul{\ol{F}}\rangle\langle\ul{G\ol{F}}\rangle + \frac{\sigma}{N} \bE \underline{G\ol{F}G^{\intercal}} -\hat{K}^{(3)}-\hat{L}^{(3)}\Big)\eqd \frac{\theta}{\hat{T}}\,,
	\end{equation}
	where $\hat T \deq -z-\bb E \ul{G}-\sigma\bb E \ul{\ol{F}}$. While we can use Lemmas \ref{prop4.4}-\ref{prop_prec} to show that $|\theta|\prec N^{\alpha-\chi}$, the value of $|\hat{T}^{-1}|$ is destructive. Instead of \eqref{T5.433}, the local semicircle law gives 
	\begin{equation*}
	\hat{T}\sim O_\prec(\eta+1-\sigma+|E-E'|)\,.
	\end{equation*}
	When $1-\sigma, |E-E'|\sim \eta$, we have $|\hat{T}|^{-1}\sim \eta^{-1}$. Thus \eqref{T5.47} can only give us $\bb E \ul{G\ol{F}}=O_{\prec}(N^{2\alpha-\chi})$, which is even worse than the trivial bound $O_{\prec}(N^{\alpha})$ from Lemma \ref{Tlem3.8}.
\end{remark}

\appendix 
\section{Proof of Lemma \ref{Tlem3.8}} \label{TA}
	We proceed by induction. By Lemma \ref{prop4.4} we see that \eqref{T3.144} is true for $k=1$. 
	
	Suppose \eqref{T3.144} is true for $k\le n-1$, and we would like to prove it for $k=n$. We split according to the parity of $n$.
	
	(i) When $n$ is odd, we write $n=2p+1$. Suppose for some $\lambda>0$ we have
	\begin{equation} \label{T3.15}
	\big|\big(G^{(1)}\cdots G^{(2p+1)}\big)_{ij} \big|  \prec \lambda
	\end{equation}
	uniformly in $i,j$ and $G^{(1)},\dots,G^{(2p+1)}\in \cal G$. Pick $G^{(1)},\dots,G^{(2p+1)}\in \cal G$, and by Cauchy-Schwarz inequality we have
	\begin{equation} \label{xiaofeizhe}
	\begin{aligned}
	\big|\big(G^{(1)}\cdots G^{(2p+1)}\big)_{ij} \big|&\leq \sum_{l}\big|\big(G^{(1)}\cdots G^{(p+1)}\big)_{il} \big|\cdot \big|\big(G^{(p+2)}\cdots G^{(2p+1)}\big)_{lj} \big|\\
	&\leq \big[\big(G^{(1)}\cdots G^{(p+1)}G^{(p+1)*}\cdots G^{(1)*}\big)_{ii}  \big(G^{(2p+1)*}\cdots G^{(p+2)*}G^{(p+2)}\cdots G^{(2p+1)}\big)_{jj}\big]^{1/2}	\,,
	\end{aligned}	
	\end{equation}
	where we abbreviate $G^{(m)*}\deq \big(G^{(m)}\big)^{*}$\,.
	Note that resolvent identity and \eqref{T3.15} shows
	\begin{equation*}
	\big(G^{(1)}\cdots G^{(p+1)}G^{(p+1)*}\cdots G^{(1)*}\big)_{ii} = \frac{1}{2\eta} \big|\big(G^{(1)}\cdots G^{(p)}(G^{(p+1)}-G^{(p+1)*})G^{(p)*}\cdots G^{(1)*}\big)_{ii} \big| \prec \frac{\lambda}{\eta}\,,
	\end{equation*}	
	and using \eqref{T3.144} for $k=n-1=2p$ shows
	\begin{equation*} 
	\big(G^{(2p+1)*}\cdots G^{(p+2)*}G^{(p+2)}\cdots G^{(2p+1)}\big)_{jj} \prec \frac{1}{\eta^{2p-1}}\,.
	\end{equation*}
	Thus we have
	\begin{equation} \label{T3.19}
	\big|\big(G^{(1)}\cdots G^{(2p+1)}\big)_{ij} \big| \prec \Big(\frac{\lambda}{\eta^{2p}}\Big)^{1/2}
	\end{equation}
	provided $\big|\big(G^{(1)}\cdots G^{(2p+1)}\big)_{ij} \big|  \prec \lambda$. The proof then follows from the trivial bound $\big|\big(G^{(1)}\cdots G^{(2p+1)}\big)_{ij} \big|  \prec  1/\eta^{2p+1}$ and iterating \eqref{T3.19}.
	
	(ii) When $n$ is even, we write $n=2p$. Pick $G^{(1)},\dots,G^{(2p)}\in \cal G$, and similar as in \eqref{xiaofeizhe} we have
	\begin{equation} \label{heshicaiyou}
	\big|\big(G^{(1)}\cdots G^{(2p)}\big)_{ij} \big|
	\leq \big[\big(G^{(1)}\cdots G^{(p)}G^{(p)*}\cdots G^{(1)*}\big)_{ii}  \big(G^{(2p)*}\cdots G^{(p+1)*}G^{(p+1)}\cdots G^{(2p)}\big)_{jj}\big]^{1/2}	\,.
	\end{equation}
	Using resolvent identity and \eqref{T3.144} for $k=n-1=2p-1$ we have
	\begin{equation*}
	\big(G^{(1)}\cdots G^{(p)}G^{(p)*}\cdots G^{(1)*}\big)_{ii} = \frac{1}{2\eta} \big|\big(G^{(1)}\cdots G^{(p-1)}(G^{(p)}-G^{(p)*})G^{(p-1)*}\cdots G^{(1)*}\big)_{ii} \big| \prec \frac{1}{\eta^{2p-1}}\,,
	\end{equation*}
	combining with a similar estimate of the last factor on RHS of \eqref{heshicaiyou} we get
	\begin{equation*}
	\big|\big(G^{(1)}\cdots G^{(2p)}\big)_{ij} \big| \prec \frac{1}{\eta^{2p-1}}
	\end{equation*}
	as desired. 

\section{Proof of Lemma \ref{Tlem3.9}} \label{TB}
	The proof of \eqref{T3.16} is similar to that of Lemma 4.8 in \cite{HK}. By the resolvent identity and Lemma \ref{lem:3.1}, we arrive at
	\begin{equation} \label{eqn: 2.59}
	\bE \underline{G^2} = \frac{1}{T} \bE \underline{G} + \frac{2}{T} \bE \langle \underline{G} \rangle \langle \underline{G^2} \rangle + \frac{2\sigma}{TN} \bE \underline{G^2G^{\intercal}} -\frac{K^{(5)}}{T}- \frac{L^{(5)}}{T}\,,
	\end{equation} 
	where $T \deq -z-2\bb E \ul{G}$
	\begin{equation*}
	K^{(2)}=N^{-2} \sum\limits_{i} \bE\frac{\partial (G^2)_{ii}}{\partial H_{ii}}(\zeta_i-1-\sigma)\,,
	\end{equation*}
	and
	\begin{equation} \label{eqn: 2.60}
	L^{(2)}= \frac{1}{N} \sum\limits_{i,j} \Bigg[ \sum\limits_{k=2}^{\ell} \sum_{\substack{p,q \ge 0,\\p+q=k}}\frac{1}{p!\,q!} \mathcal{C}_{p+1,q}(H_{ji}) \bE \frac{\partial^k (G^2)_{ij}}{\partial H_{ji}^p \partial H_{ij}^q} +R_{\ell+1}^{(2,ji)}  \Bigg]\,.
	\end{equation} 
	Here $R_{l+1}^{(2,ji)}$ is a remainder term defined analogously to $R_{\ell+1}^{(ji)}$ in (\ref{tau}). By Lemmas \ref{prop4.4}-\ref{prop_prec}, we can argue similarly as in the proof of Lemma 4.8 in \cite{HK} and show that every term on the RHS of \eqref{eqn: 2.59} is bounded by $O_{\prec}(N^{\alpha-\chi})$. This proves \eqref{T3.16}.
	
	The proof of \eqref{T3.17} is similar to that of Lemma 4.3 (ii) in \cite{HK}. By the resolvent identity and Lemma \ref{lem:3.1}, we have
	\begin{equation} \label{3.76}
	\bE \underline{G} = \frac{1}{U} \Big( 1+\bE\langle \underline{G}\rangle^2+\frac{\sigma}{N}\bE\underline{G^2}-K^{(6)}-L^{(6)}\Big)\,,
	\end{equation}
	where $U\deq -z-\bE \underline{G}$,
	\begin{equation*}
	K^{(6)}=N^{-2} \sum\limits_{i} \bE\frac{\partial G_{ii}}{\partial H_{ii}}(\zeta_i-1-\sigma)\,,
	\end{equation*}
	and
	\begin{equation*} 
	L^{(6)}= \frac{1}{N}\sum\limits_{i,j} \Bigg[ \sum\limits_{k=2}^{\ell} \sum_{\substack{p,q \ge 0,\\p+q=k}}\frac{1}{p!\,q!} \mathcal{C}_{p+1,q}(H_{ji}) \bE \frac{\partial^k G_{ij}}{\partial H_{ji}^p \partial H_{ij}^q} +R_{\ell+1}^{(6,ji)}  \Bigg]\,.
	\end{equation*}
	Here $R_{\ell+1}^{(6,ji)}$ is a remainder term defined analogously to $R_{\ell+1}^{(ji)}$ in (\ref{tau}). By Lemmas \ref{prop4.4}-\ref{prop_prec}, we can argue similarly as in the proof of Lemma 4.3 (ii) in \cite{HK} and show that the last four terms on the RHS of \eqref{3.76} is bounded by $O(N^{\alpha-1-\chi})$. Thus
	\begin{equation} \label{3.78}
	\bE \underline{G}(z+\bE \underline{G})+1=O_{\prec}(N^{\alpha-1-\chi})\,.
	\end{equation}
	Again by the argument in the proof of Lemma 4.3 (ii) in \cite{HK} we arrive at
	$$
	|\bE  \underline{G} -m(z)|=O_{\prec}(N^{\alpha-1-\chi})\,,
	$$
	which completes the proof.

{\small
	
	\bibliography{bibliography} 
	
	\bibliographystyle{amsplain}
}
\end{document}